\documentclass{amuc}
\usepackage{latexsym}
\usepackage{empheq}
\usepackage[colorlinks=true,citecolor=blue]{hyperref}  
\usepackage{amssymb,amsmath,amsfonts}
\usepackage{mathtools}

\def\O{\Omega}

\numberwithin{equation}{section}
\newtheorem{theorem}{Theorem}[section]

\newtheorem{lemma}[theorem]{Lemma}

\newtheorem{corollary}[theorem]{Corollary}
\theoremstyle{definition}
\newtheorem{definition}[theorem]{Definition}
\newtheorem{example}{Example}

\newtheorem{remark}{Remark}

\usepackage[utf8]{inputenc}

\newtheoremstyle{case}{}{}{}{}{}{:}{ }{}
\theoremstyle{case}

\usepackage{algorithm}
\usepackage[noend]{algpseudocode}

\usepackage{multirow, array, booktabs}
\usepackage{graphics}
\usepackage{epstopdf}
\usepackage{diagbox}
\usepackage{adjustbox,xcolor}
\usepackage{makecell}
\DeclareMathOperator*{\loglog}{loglog}
\DeclareMathOperator*{\trapz}{trapz}

\usepackage{enumitem}
\makeatletter
\newcommand{\mylabel}[2]{#2\def\@currentlabel{#2}\label{#1}}
\makeatother

\makeatletter
\def\BState{\State\hskip-\ALG@thistlm}
\makeatother

\def\subjclassname{{\rm 2000} {\it Mathematics Subject Classification}}
\numberwithin{equation}{section}

\begin{document}

\title[Approximating solutions of nonlinear Operators in Lebesgue spaces]{New Method for Computing zeros of monotone maps in Lebesgue spaces with applications to integral equations, fixed points, optimization, and variational inequality problems} 

\author{A. U. Bello}
\address{Institute of Mathematics,
African University of Science and Technology, Abuja, Nigeria}
\email{uabdulmalik@gmail.com}


\author{M. O. Uba}
\address{Department of Mathematics, University of Nigeria, Nsukka, Nigeria}
\email{markjoeuba@gmail.com}

\author{M. T. Omojola}
\address{Department of Mathematical Sciences,
Federal University of Technology, Akure, Nigeria}
\email{omojolamts116762@futa.edu.ng}

\author{M. A. Onyido}
\address{Department of Mathematical Sciences, 
Northern Illinois University, USA}
\email{mao0021@auburn.edu}

\author{C. I. Udeani}
\address{Department of Mathematics and Statistics,
Comenius University in Bratislava, Slovakia}
\email{cyrilizuchukwu04@gmail.com}

\keywords{Monotone Operators, Strong Convergence, $L_p$ Spaces, Hammerstein integral equations, fixed points, Convex optimization, variational inequality.\\Corresponding author: markjoeuba@gmail.com}

{47H05, 47J25}

\maketitle

\begin{abstract}
 Let $E = L_p, \; 1<p\leq 2,$ and $A : E \to E^*$ be a bounded monotone map such that $0 \in R(A)$. In this paper, we introduce and study an algorithm for approximating zeros of $A$. Furthermore, we study the application of this algorithm to the approximation of Hammerstein integral equations, fixed points, convex optimization, and variational inequality problems. Finally, we present numerical and illustrative examples of our results and their applications.
\end{abstract}
 \maketitle 
 \section{Introduction}
\noindent Let $H$ be a real inner product space and $A \subset{H\times H}$ be a multivalued operator. A is called monotone if
$$\langle u-v, x-y \rangle \geq 0$$
for each $(x,u), (y,v) \in A.$ A is called maximal monotone if it is not properly contained in any monotone subset of $H \times H$. If the operator $A$ is single valued, then $A$ is monotone if for each $x, y \in H,$
$$\langle Ax-Ay, x-y \rangle \geq 0.$$

\noindent Let $E$ be a real normed space and $A : E \to E$ be a map. $A$ is called {\it accretive} if for each $x, y \in E$ there exists $j(x-y) \in J(x-y)$ such that
       $$\langle Ax - Ay,  j(x-y)\rangle \geq 0,$$
 where $J$ is the normalized duality mapping on $E$ defined by

\begin{eqnarray}
 Jx :=\{f^*\in E^*: \langle x,f^*\rangle=||x||^2 = ||f^*||^2\}.
\end{eqnarray}
In Hilbert spaces, $J$ is the identity map on $H$. Hence, in Hilbert spaces, accretive operators are monotone.
\noindent Let $A : E \to E^*$ be an operator. A is called {\it monotone} if for each $x, y \in E$,
   $$\langle Ax - Ay,  x-y\rangle \geq 0.$$
A is called maximal monotone if it is monotone and  $R(J + \lambda A) = E^*,$ for all $\lambda \ge 0$.\\

\noindent The study of monotone operators have attracted significant attention owing to their numerous application and intimate connection with nonlinear evolutionary equations. Some authors used the concept of monotone operators to establish the existence of periodic solutions to a general class of nonlinear evolution equations in Hilbert spaces \cite{br15}.\\

\noindent Monotone operators have also been found to have strong connection with optimization. For example, let $E$ be a real normed space and $f : E \to R \bigcup \{ +\infty \}$ be a proper lower semi-continuous and convex  function. The {\it subdifferential} of $f$, $\partial f : E \rightrightarrows E^*$ is defined by
\begin{eqnarray*}
 \partial f(x)=\big \{x^*\in E^* : f(y)-f(x)\geq \big <y-x,x^*\big >, ~~\forall~y\in E\big \},~ x \in E.
\end{eqnarray*}

\noindent It is well known that $\partial f$ is a {\it monotone operator} on $E$, and that {\it $0\in \partial f(x)$ if and only if $x$ is a minimizer of $f$}. Setting $\partial f \equiv A$, it follows that solving the inclusion $0\in Au$ is solving for a minimizer of $f$. It is also well-known that any maximal monotone map $A : \mathbb{R} \rightrightarrows \mathbb{R}$ is the subdifferential of a proper, convex, and lower semi-continuous function \cite{Io}.\\

\begin{itemize}
    \item The extension of the monotonicity definition to operators from a Banach space into its dual has been the starting point for the development of nonlinear functional analysis $\dots$. The monotone maps constitute the most manageable class because of the very simple structure of the monotonicity condition. The monotone mappings appear in a rather wide variety of contexts since they can be found in many functional equations. They can also be found in calculus of variations as subdifferential of convex functions (\cite{pas}, p. 101).
\end{itemize}

\noindent A fundamental problem of great interest, which has variety of applications, in the study of monotone maps in Banach spaces is the following:

\begin{itemize}
    \item Let $E$ be a real Banach space. Find $u \in E$ such that
\begin{eqnarray}\label{zero}
 0 \in Au,
\end{eqnarray}
      where $A : E \rightrightarrows E^*$ is a monotone-type operator.
\end{itemize}

\noindent There are several studies on existence theorems for $Au = 0$, where $A$ is of the monotone-type (see e.g., \cite{pas}). Moreover, for the approximation of zeros of the equation $Au = 0$, several convergence results (see e.g., \cite{cec1, cuuoi, mart, r2, rs, rock, sol, xuR, xuj, xuh}) have been studied. One important classical way of approximating solution of (\ref{zero}) is the so-called {\it proximal point algorithm (PPA)} introduced and studied in \cite{mart}, which was further studied by author of \cite{rock} and a host of other authors. Several attempts have been made to modify the PPA to obtain strong convergence theorems. One of the immediate modification of PPA is presented in \cite{sol}, where the authors proved strong convergence result for PPA. 

\noindent The following important result was obtained in \cite{Bruck}:

\begin{theorem}\cite{Bruck}.
Let $U$ be a maximal monotone operator on $H$ with $0 \in R(A)$. Suppose $\{\lambda_n\}$ and $\{\theta_n \}$ are acceptably paired, $z \in H,$ and the sequence $\{x_n\} \subset{D(U)}$ satisfies
\begin{eqnarray}\label{bruc}
 x_{n+1} = x_n - \lambda_n(v_n + \theta_n(x_n -z)),~~ v_n \in U(x_n), n \geq 1.
\end{eqnarray}
If $\{x_n\}$ and $\{v_n\}$ are bounded, then $\{x_n\}$ converges strongly to $x^*$ a point of $U^{-1}(0)$ closest to $z$.
\end{theorem}

\noindent We make the following remarks:
\begin{itemize}
    \item Setting $\theta_n \equiv 0$ in (\ref{bruc}), the recursion formula reduces to 
    $$x_{n+1} = x_n - \lambda_nv_n,$$
which has been shown by some authors to converge to a zero of $U$  when $U$ is strongly monotone. However, when $U$ is monotone, in general, it does not converge.
\item The algorithm in \cite{Bruck} is superior to those in \cite{mart} and \cite{sol}, for the following reasons: The recursion formula is explicit, i.e., there is no need of solving an equation at each step of the iteration; there is no computation of sets (which involve projections) at each step; finally, the recursion formula is inverse-computation free.
    \end{itemize}

\noindent There are several studies on extensions of the result of \cite{Bruck} in more general Banach spaces when the operator $U : E \to E$ is accretive (see e.g., \cite{cec1, ceal, r2, rs, xuR, xuj, xuh}). Accretive operators are easier to handle when it comes to the approximation of zeros than their monotone counterpart mainly because of the presence of the normalized duality mapping in their definition. Moreover, many inequalities involving duality mappings in Banach spaces are readily available for accretive operators (see e.g., \cite{xuh}). However, a little has been done on the approximation of monotone operators in more general Banach spaces.\\

\noindent The successes achieved using geometric properties developed from the mid 1980s to early 1990s in approximating zeros of {\it accretive-type mappings} has not been carried over to approximating zeros of {\it monotone-type mappings}. The first problem is that since $A$ maps $E$ to $E^*$, for $x_n \in E,\; Ax_n$ is in $E^*$. Consequently, a recursion formula involving only $x_n$ and $Ax_n$ may not be well-defined. Another difficulty is that the normalized duality map that appears in most of the Banach space inequalities developed and in the definition of accretive-type mappings does not appear in the definition of monotone-type mappings in general Banach spaces. This creates a number of difficulties. Thus, several attempts have been made to overcome these difficulties (see e.g., \cite{id, b2, rs, zeg})). A typical recent result in this direction is the result of \cite{rs}. Before stating the result, we consider the following algorithm:

\begin{equation}\label{reich}
\begin{cases}
    & x_{0} \in X; \cr
                     & \eta_n^i =\xi_n^i + \frac{1}{\lambda_n^i}(\triangledown f(y_n^i) - \triangledown f(x_n)), \xi_n^i \in A_iy_n^i;\cr
                    & H_n^i =\{z \in Z :\langle \xi_n^i, z-y_n^i\rangle\leq0\};\cr
                    &H_n = \bigcap_{i=1}^n H_n^i;\cr
                    & W_n =\{z\in X : \langle \triangledown f(x_0) - \triangledown f(x_n), z-x_n  \rangle \leq 0;\cr
                     & x_{n+1} = proj_{H_n\bigcap W_n}^f(x_0),\; n \geq 0,
      \end{cases}
    \end{equation}

\noindent where $f:X \to \mathbb{R} \bigcup \{+\infty \}$ is a proper, lower semi-continuous and convex function; $\triangledown f(x)$ is the gradient of $f$ at $x$; $proj_{H_n\bigcap W_n}^f(x_0)$ is the generalized projection of $x_0$ onto the set $H_n\bigcap W_n$ with respect to the Bregman distance associated with $f$. We now give the statement of their result.

\begin{theorem}\cite{rs}\label{rch}.
Let $X$ be a real reflexive space and $A_i : X \to 2^{X^*}, \; i = 1,2,\dots,N$ be $N$ maximal monotone operators such that $Z= \bigcap_{i=1}^NA_i^{-1}(0^*) \neq \emptyset$. Let $f:X \to \mathbb{R}$ be a Legendre function, which is bounded, uniformly Frechet differentiable, and totally convex on bounded subsets of $X$. Then, for each $x_0 \in X,$ there are sequences $\{x_n\}$ that satisfy \ref{reich}. If, for each $i = 1,2,3, \cdots, \; \underset{n \to \infty}{liminf}\lambda_n^i > 0,$ and the sequences of errors $\{\xi_n^i\}_{n\in \mathbb{N}} \subset{X^*}$ satisfy $\underset{n \to \infty}{limsup}\langle \eta_n^i, y_n^i \rangle \leq 0$,  then each such sequence $\{x_n\}$ converges strongly to $proj_Z^f(x_0)$ as $n \to \infty$.
\end{theorem}

\noindent If $i = 1$ in theorem \ref{rch}, then the algorithm reduces to approximation of a zero of single maximal monotone operator. We also remark that, though the approximation method used by authors in \cite{rs} yields strong convergence, one has to compute some sets involving generalized projections at each step of the algorithm. Moreover, the gradient of $f$ is not easy to compute in an arbitrary Banach space, and the six-step algorithm is not easy to implement in any possible application. Thus, there is need for a more efficient algorithm for monotone operators in general Banach spaces. \\

\noindent In this paper, we introduce and study an explicit algorithm for approximating a solution of a maximal monotone equation in $L_p$ spaces ($1<p\le2$). We proved that the sequence generated by the algorithm converges strongly to
a solution of the equation. The result is applied to approximate a solution of a convex optimization problem, $J-$fixed points problem, variational inequality problem, and a solution of a Hammerstein integral equation. Finally, numerical and illustrative examples
are given to illustrate the implementability and applicability of our algorithm for approximating a solution of a convex optimization problem, a maximal monotone equation, and a Hammerstein integral equation.\\

\noindent We remark that, in Hilbert spaces (setting $z=0$ and $A$ as single valued in the algorithm in \cite{Bruck}), our recursion formula agrees with that in \cite{Bruck}. Additionally, our algorithm is a one-step algorithm and does not require computation of sets at each step. The result is applicable to the class of maximal monotone operators. Finally, we confined our result to $L_p$ spaces ($1<p\leq 2$) for two main reasons. First, $L_p$ spaces are the largest class for which the precise formula for computing the normalized duality map $J$ is known (see e.g., \cite{yak}). Second, it is well-known that if meas$(\Omega) < \infty$, then for $p \leq q, \; L_q(\Omega) \subseteq L_p(\Omega).$

\section{ Preliminaries}
\noindent We begin by giving the following definitions and lemmas, which are key to the proofs of our main results.\\

\noindent Let $\{\alpha_n\}$ and $\{\theta_n\}$ be two nonnegative real sequences. Then, $\{\alpha_n\}$ and $\{\theta_n\}$ are said to be acceptably paired if $\{\theta_n\}$ is decreasing, $\underset{n\to\infty}{lim} \; \theta_n = 0$ and there exists a strictly increasing sequence of positive integers $\{n(i)\}$ such that 

\begin{eqnarray}\label{2.1}
 \underset{i \to \infty}{lim}\sum_{j=n(i)}^{n(i+1)}\alpha_j^2 = 0,
\end{eqnarray}

\begin{eqnarray}\label{2.2}
 \underset{i \to \infty}{liminf}\;\theta_{n(i)} \sum_{j=n(i)}^{n(i+1)}\alpha_j > 0,
\end{eqnarray}

\begin{eqnarray}\label{2.3}
 \underset{i \to \infty}{liminf} (\theta_{n(i)} - \theta_{n(i+1)}) \sum_{j=n(i)}^{n(i+1)}\alpha_j = 0.
\end{eqnarray}
For more details and examples of acceptably paired sequences see Bruck \cite{Bruck}.\\

\noindent Let $E$ be a smooth real Banach space with dual $E^*$. The function $\phi:E\times E\to\mathbb{R}$, defined by:
\begin{eqnarray}\label{Lya}
 \phi(x,y)=\|x\|^2-2\langle x,Jy\rangle+\|y\|^2,~~\text{for}~x,y\in E,
\end{eqnarray}
where $J$ is the normalized duality mapping from $E$ into $E^*$, introduced in \cite{b1}
has been studied by several authors (see e.g., \cite{b2, r1}). If $E=H$, a real Hilbert space, then
equation (\ref{Lya}) reduces to $\phi(x,y)=\|x-y\|^2$ for $x,y\in H.$ It is obvious from the definition of the function $\phi$ that
\begin{eqnarray}\label{2.5}
 (\|x\|-\|y\|)^2\leq \phi(x,y)\leq(\|x\|+\|y\|)^2~~\text{for}~x,y\in E.
\end{eqnarray}
\noindent Define $V:E\times E^*\to \mathbb{R}$ by
\begin{eqnarray}
V(x,x^*)=\|x\|^2-2\langle x,x^*\rangle+\|x^*\|^2.
\end{eqnarray}
Then, it is easy to see that 
\begin{eqnarray}\label{22.3}
 V(x,x^*)=\phi(x,J^{-1}(x^*)),~\forall~ x\in E,~x^*\in E^*.\label{v}
\end{eqnarray}
\noindent A map $\Pi_k : E \to K$ defined by $\Pi_k(x) = \Bar{x}$, where $\Bar{x}$ is the solution of 
$$min\{\phi(x,y), \; y \in K\},$$
is called the {\it generalized projection map}. It is well known that if $\Bar{x} =\Pi_k(x)$, then
$$\langle J(x) - J(\Bar{x}), y - \Bar{x} \rangle \leq 0, \; \forall \; y \in K.$$

\begin{definition}{\bf (J-fixed point)}\cite{coeu}
Let $E$ be a arbitrary normed space and $E^*$ be its dual. Let $T:E \to E^*$ be any mapping. A point $x^*\in E$ is called a {\it $J$-fixed point of $T$} if $Tx^*=Jx^*$. The set of $J$-fixed points of $T$ will be denoted by $F_{J}(T)$.
\end{definition}

\begin{definition}{\bf (J-pseudocontractive mappings)}\cite{coeu}
Let $E$ be a arbitrary normed space and $E^*$ be its dual. A mapping $T:E \to E^*$
is called $J$-pseudocontractive if for every $x, y \in E$, $\langle Tx - Ty, x - y \rangle \leq \langle Jx - Jy, x - y \rangle.$
\end{definition}

\begin{lemma}\label{alb}\cite{b1}
 Let $E$ be a reflexive striclty convex and smooth Banach space with $E^*$ as its dual. Then, 
 \begin{eqnarray}
  V(x,x^*)+2\langle J^{-1}x^*-x,y^*\rangle\leq V(x,x^*+y^*).\label{v2}
 \end{eqnarray}
for all $x\in E$ and $x^*,y^*\in E^{*}.$
\end{lemma}

\begin{lemma}\cite{cec2} \label{chid}
Assume $1<p<\infty$ and $t_p$ is the unique solution of the equation
$$(p-1)t^{p-1} + (p-1)t^{p-2} - 1 = 0, \; 0 < t < 1.$$ Let $c_p = (1+t_p^{p-1})(1+t_p)^{-(p-1)}.$ Then, we have the following:\\
If $1<p\le2,$ then for all $x, y$ in $L_p$,
\begin{eqnarray}\label{2.9}
 \|x + y\|^2 \geq \|x\|^2 + 2\langle y, j(x) \rangle + c_p\|y\|^2,
\end{eqnarray}
\begin{eqnarray}\label{2.10}
 \langle x- y, J(x) - J(y) \rangle \geq (p-1)\|x -y\|^2.
\end{eqnarray}
\end{lemma}
\noindent Observe that this inequality yields
\begin{eqnarray}\label{2.11}
 \|J^{-1}(x) - J^{-1}(y)\| \leq L\|x - y \|,
\end{eqnarray}
where $L := \frac{1}{p-1}.$

\begin{remark}
 It is clear from inequality (\ref{2.9}) that since $c_p \geq 1,$ we have 
 $$ \|x + y\|^2 \geq \|x\|^2 + 2\langle y, j(x) \rangle + \|y\|^2.$$
 Interchanging $x$ and $y$, we obtain
 $$ \|x + y\|^2 \geq \|y\|^2 + 2\langle x, j(y) \rangle + \|x\|^2.$$ Replacing $x$ by $-x$, we obtain
  $$\|x - y\|^2 \geq \|x\|^2 - 2\langle x, j(y) \rangle + \|y\|^2 = \phi(x,y).$$
  So that 
 \begin{eqnarray}\label{2.12}
  \|x - y\|^2 \geq \phi(x, y), \; \forall \; x, y \in E.
 \end{eqnarray}
 
\end{remark}
\begin{lemma}\cite{b2}\label{katak}
Let $E$ be a real smooth and uniformly convex Banach space and let $\{x_{n}\}$ and $\{y_{n}\}$ be sequences in $E.$  If either  
$\{x_{n}\}$ or $\{y_{n}\}$ is bounded and $\phi(x_{n},y_{n}) \to 0$ as $n \to \infty$, then  $||x_{n}-y_{n}|| \to 0$ as $n \to \infty$.
\end{lemma}

\begin{lemma} \cite {r2}\label{reich2} 
 Let $E$ be a real Banach space such that $E^*$ is strictly convex and has a  Fr\'{e}tchet differentiable norm. Let $A:E \to E^*$ be a maximal monotone operator. Then for arbitrary $z \in E^*, \; J_{\lambda} = (J+ \lambda A)^{-1}z$ converges strongly to $Rz$ as $\lambda \to \infty,$ where $Rz$ is the unique point of $A^{-1}(0)$ satisfying
 $$\langle z - J(Rz), Rz - y\rangle \geq 0 \; \forall \; y \in A^{-1}(0).$$
 
\end{lemma}

\begin{remark}
It is clear that if we replace $\lambda$ with $\theta_n$ in Lemma \ref{reich2}, where $\{\theta_n\}$ is a real sequence in $(0,1)$ such that $\theta_n \to 0$ as $n \to \infty$ and by setting 
\begin{eqnarray}\label{2.13}
 y_n = \Bigg(J + \frac{1}{\theta_n}A \Bigg)^{-1}z,
\end{eqnarray}
we have $y_n \to Rz$ as $n \to \infty$. In particular, taking $z = 0 \in E^*$ we have for each $n \ge 1$ there exists a unique $y_n$ such that
\begin{eqnarray}\label{2.14}
 \theta_nJy_n + Ay_n = 0, \;\; \forall \; n \ge 1.
\end{eqnarray}
\end{remark}
\vskip 0.5truecm
\section{Main result}

\noindent In the sequel $\{y_n\}$ is the sequence satisfying equation (\ref{2.14}) for each $n \ge 1.$
\begin{theorem}\label{thm3.1}
Let $E = L_p, \; 1<p\leq 2.$ Let $A : E \to E^*$ be a maximal monotone and bounded map such that $0 \in R(A).$ For arbitrary $x_1 \in E,$ define the iterative sequence $\{x_n\}_{n=1}^{\infty}$ by:
\begin{eqnarray}\label{3.1}
 x_{n+1} = J^{-1}(Jx_n - \alpha_nAx_n - \alpha_n\theta_nJx_n), \; n \ge 1,
\end{eqnarray}
where $\{ \alpha_n\}_{n=1}^{\infty}$ and $\{\theta_n\}_{n=1}^{\infty}$ are real sequences in $(0,1)$ such that $\theta_n \to 0$ as $n \to \infty.$ Then there exists $\gamma > 0$ such that whenever $\alpha_n \leq \gamma \theta_n, \; \forall~ n \ge 1,$ the sequences $\{ x_n \}_{n=1}^{\infty}$ is bounded.
\end{theorem}

\begin{proof}
Let $x^* \in A^{-1}(0),$ and $r > 0$ be such that 
\begin{eqnarray}
 r > max \Big\{ \phi(x^*, x_1), \frac{4(\sqrt{r} + 2\|x^*\|)\|x^*\|}{(p-1)}\Big \}.
\end{eqnarray}
Since A is bounded, define:
$$ M_0 := 2Lsup\big\{ \|Ax + \lambda Jx\|^2 : \|x\| \leq \|x^*\| + \sqrt{r}; \; \lambda \in (0,1) \big \} + 1; \; \gamma := \frac{(p-1)r}{2(M_o + 1)}, $$
where $L > 0$ is the Lipshitz constant of $J^{-1}$ in (2.11).\\

\noindent We shall use induction to show that $\phi(x^*, x_n) \leq r, \; \forall \; n \ge 1.$ By construction $\phi(x^*, x_1) \leq r.$  Suppose that $\phi(x^*, x_n) \leq r$ up to some $n \ge 1.$ We prove that $\phi(x^*, x_{n+1}) \le r.$ Using recursion formula (\ref{3.1}) and Lemma \ref{alb} with $y^* = \alpha_nAx_n + \alpha_n \theta_n Jx_n$ we obtain
\begin{eqnarray*}\label{3.9}
    \phi(x^*, x_{n+1}) & = & 
    \phi(x^*, J^{-1}(Jx_n - \alpha_n Ax_n - \alpha_n \theta_n Jx_n))\nonumber\\
    & = & V(x^*, Jx_n - \alpha_nAx_n - \alpha_n \theta_nJx_n))\nonumber\\
    & \leq & V(x^*, Jx_n) - 2\langle J^{-1}(Jx_n - \alpha_n Ax_n - \alpha_n \theta_n Jx_n)
     -x^*, \alpha_nAx_n + \alpha_n \theta_nJx_n \rangle \nonumber \\ 
    &=&  \phi(x^*, x_n) - 2\alpha_n\langle x_n - x^*, Ax_n + \theta_n Jx_n \rangle\nonumber\\ 
    & & - 2\alpha_n\langle J^{-1}(Jx_n - \alpha_n Ax_n - \alpha_n \theta_n Jx_n)
     - J^{-1}(Jx_n), Ax_n + \theta_n Jx_n \rangle\nonumber \\
    &\leq&  \phi(x^*, x_n) - 2\alpha_n\langle x_n - x^*, Ax_n + \theta_n Jx_n \rangle \nonumber\\ 
    & & +\; 2\alpha_n \| J^{-1}(Jx_n - \alpha_n Ax_n - \alpha_n \theta_n Jx_n) - J^{-1}(Jx_n) \| \|Ax_n + \theta_n Jx_n \|.\nonumber \\
\end{eqnarray*}
Using the fact that $A$ is monotone, and that $J$ is strongly monotone (see inequality (\ref{2.10})), we have
\begin{eqnarray*}
 \langle x_n - x^*, Ax_n + \theta_n Jx_n \rangle&=&\langle x_n - x^*, Ax_n - Ax^*  \rangle \nonumber\\
 & & + \theta_n \langle x_n - x^*, Jx_n - Jx^*\rangle + \theta_n \langle x_n -x^*, Jx^* \rangle \nonumber \\ 
 &\ge & \theta_n(p-1) \| x_n - x^*\|^2 +\theta_n \langle x_n - x^*, Jx^* \rangle.   
\end{eqnarray*}
Thus we have
\begin{eqnarray*}\label{3.10}
    \phi(x^*, x_{n+1}) 
     & \leq &  \phi(x^*, x_n) - 2\alpha_n\theta_n(p-1)\|x_n - x^*\|^2  -2\alpha_n\theta_n \langle x_n - x^*, Jx^*   \rangle  \nonumber \\
     & &  +  2\alpha_n \|J^{-1}(Jx_n - \alpha_n Ax_n - \alpha_n \theta_n Jx_n) - J^{-1}(Jx_n)\|\|Ax_n + \theta_nJx_n\|\nonumber.\\
\end{eqnarray*}

\noindent Using the fact that $J^{-1}$ is Lipschitz; $\phi(x^*,x_n) \le \|x_n - x^*\|^2$ (see inequality (\ref{2.12}),\\ and the definition of $M_0$ we obtain
\begin{eqnarray*}
   \phi(x^*, x_{n+1}) 
    &\leq & (1-2\alpha_n\theta_n(p-1))\phi(x^*, x_n) - 2\alpha_n\theta_n\langle x_n - x^*, Jx^*   \rangle + 2L\alpha_n \|Ax_n + \theta_nJx_n\|^2 \cr\\
     &\leq & (1-2 \alpha_n \theta_n  (p-1)) \phi(x^*,x_n)-2\alpha_n\theta_n (\|x_n\| + \|x^*\|)\|x^*\| + 2L\alpha_n\alpha_n M_0. \nonumber \\
\end{eqnarray*}
Using inequality (\ref{2.5}), we observe that
\begin{eqnarray*}
 \phi(x^*,x_n) < r \implies (\|x_n\| - \|x^*\|)^2 \le r \implies \|x_n\| \le \|x^*\| + \sqrt{r}.
\end{eqnarray*}
Since $\alpha_n \leq \theta_n\gamma \; \forall \; n \ge 1,$ we have
\begin{eqnarray*}
    \phi(x^*, x_{n+1}) 
     & \leq & (1 - 2\alpha_n\theta_n(p-1))\phi(x^*, x_n) + 2\alpha_n\theta_n(\sqrt{r} + 2\|x^*\|)\|x^*\| + \alpha_n\alpha_nM_0 \nonumber\\
     & \le & (1 - 2\alpha_n\theta_n(p-1))r + 2\alpha_n\theta_n(\sqrt{r} + 2\|x^*\|)\|x^*\| +\alpha_n\theta_n\gamma M_0.\nonumber \\
     &\le&(1 - 2\alpha_n\theta_n(p-1))r + \frac{\alpha_n\theta_n(p-1)r}{2} +\frac{\alpha_n\theta_n(p-1)r}{2} \nonumber\\ 
     &\le& (1 - \alpha_n\theta_n(p-1))r \le r. \nonumber\\
\end{eqnarray*}
Therefore, $\phi(x^*, x_{n+1}) \le r \; \forall \; n \ge 1.$ Thus $\{\phi(x^*, x_n)\}_{n=1}^{\infty}$ is bounded. Inequality (\ref{2.5}) implies that $\{x_n\}_{n=1}^{\infty}$ is bounded.
\end{proof}
\noindent We now give our convergence result. We shall adopt the method used by Bruck in \cite{Bruck}.

\begin{theorem}\label{thm3.2}
Let $E = L_p, 1 < p \le 2$. Let $A:E \to E^*$ be a maximal monotone and bounded map such that $0 \in R(A).$  For arbitrary $x_1 \in E$, define the iterative sequence $\{x_n\}_{n=1}^{\infty}$ by:
\begin{eqnarray}\label{3.3}
x_{n+1} = J^{-1}(Jx_n - \alpha_n Ax_n - \alpha_n \theta_nJx_n), \; n \ge 1,
\end{eqnarray}
where $\{\alpha_n\}_{n=1}^{\infty}$ and $\{\theta_n\}_{n=1}^{\infty}$ are acceptably paired sequences in $(0,1)$ such that there exists $\gamma > 0$ with $\alpha_n \le \theta_n \gamma \; \forall \; n \ge 1.$ Then, the sequence $\{x_n\}_{n=1}^{\infty}$ converges strongly to the unique point $x^* = \Pi_{A^{-1}(0)}0$ of $A^{-1}(0)$ satisfying $\langle Jx^*, x - x^* \rangle \le 0$ for all $x \in A^{-1}(0).$
\end{theorem}

\begin{proof}
Let $n \ge i \ge 1.$ Using the definition of $x_{n+1}$ and $y_n$ in (\ref{3.3}) and (\ref{2.14}), respectively, and lemma \ref{alb} with $y^* = \alpha_n Ax_n + \alpha_n \theta_nJx_n$, we have 
\begin{eqnarray*}
\phi(y_{i+1}, x_{n+1}) &=& \phi(y_{i+1}, J^{-1}(Jx_n - \alpha_n Ax_n - \alpha_n \theta_nJx_n)) \nonumber \\
&=& V(y_{i+1}, Jx_n - \alpha_n Ax_n - \alpha_n \theta_nJx_n)) \nonumber \\
&\le& V(y_{i+1}, Jx_n) - 2\langle J^{-1}(Jx_n - \alpha_n Ax_n - \alpha_n \theta_nJx_n) - y_{i+1}, \alpha_n Ax_n + \alpha_n \theta_nJx_n \rangle \nonumber \\
&=& \phi(y_{i+1}, x_n) -2\alpha_n \langle x_n - y_{i+1}, Ax_n + \theta_nJx_n \rangle \nonumber\\
& & -2\alpha_n \langle J^{-1}(Jx_n - \alpha_n Ax_n - \alpha_n \theta_nJx_n) - J^{-1}(Jx_n), Ax_n + \theta_nJx_n \rangle \nonumber\\
&\le& \phi(y_{i+1}, x_n) -2\alpha_n \langle x_n - y_{i+1}, Ax_n + \theta_{i+1}Jx_n \rangle \nonumber\\
&& +2\alpha_n(\theta_{i+1} - \theta_n)\|x_n - y_{i+1}\|\|x_n\| \nonumber\\
& & +2\alpha_n \|J^{-1}(Jx_n), Ax_n + \theta_nJx_n) - J^{-1}(Jx_n) \|\|Ax_n + \theta_nJx_n \|\nonumber\\
&\le& \phi(y_{i+1}, x_n) -2\alpha_n \langle x_n - y_{i+1}, Ax_n + \theta_{i+1}Jx_n \rangle \nonumber\\
&& +\alpha_n(\theta_{i+1} - \theta_n)M_1 + \alpha_n^2M_0, 
\end{eqnarray*}
for some $M_1 > 0.$ Using equation (\ref{2.14}), we have $Ay_{i+1} = -\theta_{i+1}Jy_{i+1}.$ Thus by the monotonicity of $A$, we have 
$$\langle x_n - y_{i+1}, Ax_n + \theta_{i+1}Jy_{i+1} \rangle \ge 0.$$ 
Using also the fact that $J$ is strongly monotone (see Lemma \ref{reich}), we have
\begin{eqnarray*}
\langle x_n - y_{i+1}, Ax_n + \theta_{i+1}Jx_n \rangle &=&  
\langle x_n - y_{i+1}, Ax_n + \theta_{i+1}Jy_{i+1} \rangle \nonumber\\
& & + \langle x_n - y_{i+1}, \theta_{i+1}Jx_n - \theta_{i+1}Jy_{i+1} \rangle \nonumber\\
& \ge& \theta_{i+1}(p-1)\|x_n -y_{i+1} \|^2.\nonumber
\end{eqnarray*}
Thus, we have
\begin{eqnarray*}
\phi(y_{i+1}, x_{n+1}) & \le& \phi(y_{i+1}, x_n) - 2\alpha_n\theta_{i+1}(p-1)\|x_n - y_{i+1}\|^2 \nonumber\\
& & + \alpha_n(\theta_{i+1} - \theta_n)M_1 + \alpha_n^2M_0.
\end{eqnarray*}
Hence,
\begin{eqnarray}\label{3.4}
\phi(y_{i+1}, x_{n+1}) & \le& e^{-2\alpha_n\theta_{i+1}(p-1)}\phi(y_{i+1}, x_{n})\nonumber\\
&& +\alpha_n(\theta_{i+1} - \theta_n)M_1 + \alpha_n^2M_0.
\end{eqnarray}
Here, we have used the fact that $(1-2\alpha_n \theta_{i+1}(p-1)) \le e^{-2\alpha_n\theta_{i+1}(p-1)}.$\\
Applying induction to (\ref{3.4}) and observing that $e^{-2\alpha_n\theta_{i+1}(p-1)} \le 1$, we obtain
\begin{eqnarray}\label{3.5}
\phi(y_{i+1}, x_{n+1}) & \le& e^{-2\alpha_n\theta_{i+1}(p-1)}\sum_{j = i+1}^{n}\alpha_j\phi(y_{i+1}, x_{i+1})\nonumber\\
&& + M_1\sum_{j = i+1}^{n}\alpha_j(\theta_{i+1} - \theta_j) + M_0 \sum_{j = i+1}^{n}\alpha_j^2.
\end{eqnarray}
Since $\{\theta_n\}$ is decreasing we have $\theta_{i+1} -\theta_j \le \theta_{i+1} - \theta_{n+1}$ for $j \le n,$ so that 
\begin{eqnarray}\label{3.6}
\phi(y_{i+1}, x_{n+1}) & \le& e^{-2\theta_{i+1}(p-1)}\sum_{j = i+1}^{n}\alpha_j\phi(y_{i+1}, x_{i+1})\nonumber\\
&& + M_1(\theta_{i+1} - \theta_{n+1})\sum_{j = i+1}^{n}\alpha_j + M_0 \sum_{j = i+1}^{n}\alpha_j^2.
\end{eqnarray}
We now use (\ref{3.6}) to first prove that $\underset{k\to\infty}{lim}x_{n(k)} = x^*$ (where $x^*$ is the limit of the sequence $\{y_n\}$ satisfying (\ref{2.14})). Indeed, taking $i=n(k) - 1$ and $n = n(k+1)-1$ in (\ref{3.6}) we have
\begin{eqnarray*}
\phi(y_{n(k)}, x_{k+1}) & \le& e^{-2\theta_{n(k)}(p-1)}\sum_{j = n(k)}^{n(k+1)}\alpha_j\phi(y_{n(k)}, x_{n(k)})\cdot e^{2\theta_{n(k)}\alpha_{n(k+1)}(p-1)}\nonumber\\
&& + M_1(\theta_{n(k)} - \theta_{n(k+1)})\sum_{j = n(k)}^{n(k+1)}\alpha_j + M_0 \sum_{j = n(k)}^{n(k+1)}\alpha_j^2.
\end{eqnarray*}

\noindent Since $\underset{n\to \infty}{lim}\alpha_n = 0,$ we have $\underset{k\to\infty}{lim}e^{2\theta_{n(k)}\alpha_{n(k+1)}(p-1)}=1.$ (\ref{2.1}), (\ref{2.2}), (\ref{2.3}), and  (\ref{3.6}) imply the existence of $\gamma \in (0,1)$ such that 
\begin{equation}\label{3.7}
    \underset{k\to\infty}{lim sup}\; \phi(y_{n(k)}, x_{k+1}) \le \gamma \cdot \underset{k\to\infty}{lim sup}\; \phi(y_{n(k)}, x_{n(k)}). 
\end{equation}
The fact that $\underset{k \to \infty}{lim} y_{n(k)} = x^*$ implies
$$\underset{k\to \infty}{limsup}\;\phi(y_{n(k)}, x_{n(k+1)}) = \underset{k\to\infty}{lim sup}\; \phi(x^*, x_{n(k)}) = \underset{k\to\infty}{lim sup}\; \phi(y_{n(k)}, x_{n(k)}). $$ 
Inequality (\ref{3.7}) together with the fact that $\gamma < 1$ implies that $\underset{k\to\infty}{lim}\phi(x^*, x_{n(k)}) = 0.$  Applying Lemma \ref{katak} we see that $\underset{k\to\infty}{lim}x_{n(k)} = x^*$.\\
Now, for any $n \in \mathbb{N}$ such that $n + 1 > n(1)$ choose $k$ with $n(k) \le n + 1 < n(k+1).$ Thus, taking $i = n(k) - 1$ in (\ref{3.6}) and observing that $e^{-2\alpha_n\theta_{i+1}(p-1)} \le 1$  we have 
\begin{eqnarray*}
\phi(y_{n(k)}, x_{n+1}) \le \phi(y_{n(k)}, x_{n(k)}) + M_1(\theta_{n(k)} - \theta_{n(k)})\sum_{j = n(k)}^{n(k+1}\alpha_j + M_0 \sum_{j = n(k)}^{n(k+1)}\alpha_j^2. 
\end{eqnarray*}

\noindent Since $\underset{k\to\infty}{lim}x_{n(k)} = \underset{k\to\infty}{lim}y_{n(k)},$ then (\ref{2.2}), (\ref{2.3}) together with the last inequality imply $\phi(y_{n(k)}, x_{n+1}) \to 0$ as $n, k \to \infty.$ Lemma (\ref{katak}) implies that $\|y_{n(k)} - x_{n+1}\| \to 0$ as $n, k \to \infty.$ Therefore, $x_n \to x^*.$ The proof is complete.
\end{proof}
\begin{remark}
It is well known that in a real Hilbert space $H$, the function $\phi (x,y) $ reduces to $\|x - y\|^2.$ Thus, the following corollary is immediate.
\end{remark}

\begin{corollary}\label{cor3.3}
Let $H$ be a real Hilbert space and $A: H \to H$ be a maximal monotone and bounded map such that $0 \in R(A).$ For arbitrary $x_1 \in H,$ define the iterative sequence $\{x_n\}_{n=1}^{\infty}$ by:
\begin{equation}\label{3.8}
    x_{n+1} = x_n - \alpha_nAx_n - \alpha_n\theta_nx_n, n \ge 1,
\end{equation}
where $\{\alpha_n\}_{n=1}^{\infty}$ and $\{\theta_n\}_{n=1}^{\infty}$ are acceptably paired sequences in $(0,1)$ such that there exists $\gamma > 0$ with $\alpha_n \le \theta_n \gamma, \; \forall \; n \ge 1.$ Then, the sequence $\{x_n\}_{n=1}^{\infty}$ converges strongly to the unique point of $A^{-1}(0)$ closest to $0$.
\end{corollary}

\begin{remark}
All the results presented in this paper are for single valued operators. The same method can be adopted for multivalued operators. Thus, if the operator in theorem 1 of \cite{Bruck} is single valued (with z = 0 in the recursion formula), then the theorem coincides with corollary \ref{cor3.3}.
\end{remark}
\begin{remark}
Our results hold in $L_p$ spaces, $1 < p \le 2$. Most of the properties of the normalized duality mapping, $J$ and also the geometric properties of $L_p$ spaces vary with the range of $p$. For instance, $J^{-1}$ is Lipschitz in $L_p$ spaces, $1<p \le 2$ while it is only H\'{o}lder  continuous when $p$ lies between $2$ and $\infty$ (see e.g., \cite{yak} and \cite{cec2}, p. 55). With some modifications, our result can also be proved in $L_p$ spaces, $2 \le p < \infty.$ However, we do not carry out this extension. It is also well known that if meas$(\Omega) < \infty$, then for $p \le q, \; L_q(\Omega) \subset{L_p(\Omega)}.$
\end{remark}

\begin{remark}
The recursion formula in Theorem \ref{thm3.2} is one-step and does not involve computation of sets involving generalized projections. Moreover, $L_p$ spaces are the largest spaces for which the precise value of the duality mapping is known. In fact, if $E = L_p, \; 1<p<\infty$ and $J : E \to E^*$ is the normalized duality map, then for each $f \in L_p$, 
$$J(f) = \|f\|^{p-1}\cdot sign \frac{f}{\|f\|^{p-1}}.$$
\end{remark}
\noindent A prototype for the parameters used in our theorems are
\begin{equation}\label{seq}
\alpha_n = \frac{1}{n} \;\; and \;\; \theta_n = \frac{1}{loglog n},
\end{equation}
with $n(i) = i^i.$ See \cite{Bruck} for more details and examples of acceptably paired sequences.

\section{Applications}
\subsection{Application to Hammerstein integral equations}
Let $\Omega$ be a measurable and bounded subset of ${\mathbb{R}}^n$. A nonlinear integral equation of Hammerstein  type is of the form
\begin{equation}\label{Hamm4}
u(x)+\int_\Omega k(x,y)f(y,u(y))dy=w(x),
\end{equation}
where $dy$ is a $\sigma$-finite measure. The function $k:\Omega \times \Omega \rightarrow \mathbb{R}$ is the kernel of ($\ref{Hamm4}$), and $f:\Omega \times \mathbb{R} \rightarrow \mathbb{R}$ is a measurable real-valued function. The function $w$ and the unknown function $u$ lie in a suitable Banach space of measurable real-valued functions, namely, $\mathcal{F}(\Omega, \mathbb{R})$. Suppose we define the operators $F:\mathcal{F}(\Omega, \mathbb{R})\rightarrow \mathcal{F}(\Omega, \mathbb{R})$ and $K:\mathcal{F}(\Omega, \mathbb{R})\rightarrow \mathcal{F}(\Omega, \mathbb{R})$ by 
\begin{equation}\label{Hm}
Fu(x)=f(x,u(x))\quad\textup{and}\quad Kv(x)=\int_\Omega k(x,y)v(y)dy,
\end{equation}
where $x\in \Omega.$ Then, ($\ref{Hamm4}$) can be re-written in an abstract Hammerstein equation as
\begin{equation}\label{ham6}
u+KFu=0,
\end{equation}
where, without loss of generality, $w$ is the zero map in $\mathcal{F}(\Omega, \mathbb{R})$. 

\vskip0.2truecm \noindent
Interest in Hammerstein integral equations arises primarily from the fact that various problems that originate in the form of differential equations, for example, elliptic
boundary value problems whose linear parts posses Green's functions can be transformed into the form (\ref{Hamm4}) 
(see, for instance, \cite[Chapter IV, p. 164]{pas}). Moreover, the Hammerstein equation also plays essential role in the theory of optimal control systems, automation, and network theory. For instance, in the study of automatic control using a Hammerstein model, authors in \cite{ng} suggested an iterative method for the identification of nonlinear systems, for samples of inputs and outputs in the presence of noise. See, e.g., Dolezale \cite{dole} for more details on problems in optimal control, automation and network systems that can be modeled as Hammerstein equations.
\vskip0.2truecm \noindent
Several existence and uniqueness theorems have been proved for equations of the Hammerstein type (see, for example, \cite{f9, f16}). Until now, no method, which finds the closed form solutions to these nonlinear equations, is known. Thus, iterative algorithms which estimate these solutions are of great interest 
(see e.g., \cite{f38, f31, f19, mm, uba, ubaetal, ubaonyidopun} and also Chapter 13 of \cite{cec2}).
\vskip0.2truecm
\noindent Here, we shall apply Theorem \ref{thm3.2} to approximate a solution of problem (\ref{ham6}). The following lemma and remark would be needed in what follows.
\begin{lemma}[\cite{BrMax}]\label{BrMax}
	Let $X$ be a strictly convex reflexive Banach space with a strictly convex conjugate space $X^*$, $T_1$ a maximal monotone mapping from $X$ to $X^*$, $T_2$ a hemicontinuous monotone mapping of all of $X$ into $X^*$ which carries bounded subsets of $X$ into bounded subsets of $X^*$.
	Then, the mapping $T=T_1+T_2$ is a maximal monotone mapping of $X$ into $X^*$.
\end{lemma}
\begin{lemma}[\cite{uba}]\label{Bs}
	Let $X$ be a uniformly convex and uniformly smooth real Banach space with dual space $X^*$ and $E=X\times X^*$. Let $F:X\to X^*$ and $K:X^*\to X$ be monotone mappings. Let $ A:E\to E^*$ be defined by $A([u,v])=[Fu-v,Kv+u]$. Then, $A$ is maximal monotone.
\end{lemma}
\begin{lemma} \label{normstar}\cite{uba}
  Let $X, X^*$ be uniformly convex and uniformly smooth real Banach spaces. Let
$E = X \times X^*$ with the norm $\|z\|_E = (\|u\|^2 + \|v\|^2)^{\frac{1}{2}}$
 , for any $z = [u, v] \in E$. Let $E^* = X^*\times X$
denote the dual space of E. For arbitrary $x = [x_1, x_2] \in E$, define the map $J_E : E \to E^*$ by
$$J_E(x) = J_E[x_1, x_2] := [J_X(x_1), J_{X^*} (x_2)],$$
so that for arbitrary $z_1 = [u_1, v_1],\; z_2 = [u_2, v_2]$ in $E$, the duality pairing $\langle \cdot , \cdot \rangle$ is given by
$$\langle z_1, J_E\rangle := \langle u_1, J_X(u_2)\rangle + \langle v_1, J_{X^*} (v_2)\rangle.$$
Then, $E$ is uniformly smooth and uniformly convex.   
\end{lemma}

\begin{remark}\label{z3}
	We remark that for $A$  defined in Lemma \ref{Bs}, $[u^*,v^*]$ is a zero of $A$ if and only if $u^*$ solves $(\ref{ham6})$, where $v^*=Fu^*$.
\end{remark}

\noindent We now present the following theorem.
\begin{theorem}\label{hg}
	Let $X = L_p, 1 < p \le 2$. Also, let $F:X \to X^*$ and $K:X^{*} \to X$ be maximal monotone and bounded maps. Let $E:=X\times X^*$ and $A:E\to E^*$ be defined by $A([u,v]):=[Fu-v,Kv+u]$. For arbitrary $w = (u, v) \in E$, define the iterative sequence $\{w_n\}_{n=1}^{\infty}$ by:
	\begin{eqnarray}\label{3.A3}
	w_{n+1} = J_{E^*}^{-1}(J_E w_n - \alpha_n Aw_n - \alpha_n \theta_nJ_E w_n), \; n \ge 1,
	\end{eqnarray}
	where $\{\alpha_n\}_{n=1}^{\infty}$ and $\{\theta_n\}_{n=1}^{\infty}$ are acceptably paired sequences in $(0,1)$ such that there exists $\gamma > 0$ with $\alpha_n \le \theta_n \gamma \; \forall \; n \ge 1.$ Assume that $u+KFu=0$ has a solution. Then, the sequences $\{u_n\}_{n=1}^{\infty}$ and $\{v_n\}_{n=1}^{\infty}$ converge strongly to $ u^{*}$ and $ v^{*},$ respectively, where $u^* = \Pi_{A^{-1}(0)}0$ of $A^{-1}(0)$ is a unique solution of $u+KFu=0$ satisfying $\langle Ju^*, x - u^* \rangle \le 0$ for all $x \in A^{-1}(0).$ 
\end{theorem}

\begin{proof}
	By Lemma \ref{normstar}, $E$ is uniformly convex and uniformly smooth, and by Lemma \ref{Bs}, $A$ is maximal monotone. Hence, the conclusion follows from Theorem \ref{thm3.2} and Remark \ref{z3}.
\end{proof}
\noindent Theorem \ref{hg} can also be stated as follows.
\begin{theorem}\cite{ubaetal}\label{thvm3.2}
	Let $X = L_p, 1 < p \le 2$. Also, let $F:X \to X^*$ and $K:X^{*} \to X$ be maximal monotone and bounded maps such that $0 \in R(F).$  For arbitrary $(u_{1},v_{1}) \in X\times X^{*}$, define the iterative sequences $\{u_n\}_{n=1}^{\infty}$ and $\{v_n\}_{n=1}^{\infty}$ by
	\begin{equation}\label{3v.3}
	\begin{cases}
	u_{n+1} = J^{-1}(Ju_n - \alpha_n (Fu_n-v_{n}) - \alpha_n \theta_nJu_n), \; n \ge 1,\\
	v_{n+1} = J^{-1}(Jv_n - \alpha_n (Kv_n+u_{n}) - \alpha_n \theta_nJv_n), \; n \ge 1,
	\end{cases}
	\end{equation}
	where $\{\alpha_n\}_{n=1}^{\infty}$ and $\{\theta_n\}_{n=1}^{\infty}$ are acceptably paired sequences in $(0,1)$ such that there exists $\gamma > 0$ with $\alpha_n \le \theta_n \gamma \; \forall \; n \ge 1.$ Assume that $ u+KFu=0$ has a solution. Then, the sequences $\{u_n\}_{n=1}^{\infty}$ and $\{v_n\}_{n=1}^{\infty}$ converge strongly to $ u^{*}$ and $ v^{*},$ respectively, where $u^* = \Pi_{A^{-1}(0)}0$ of $A^{-1}(0)$ is a unique solution of $u+KFu=0$ satisfying $\langle Ju^*, x - u^* \rangle \le 0$ for all $x \in A^{-1}(0).$ 
\end{theorem}

\subsection{Application to convex optimization problem}
The following lemma will be very effective in the proof of the theorem in this subsection.
\begin{lemma} [\cite{rock2}]\label{opt}
	Let $X$ be a Banach space and let $f:X\rightarrow \mathbb{R}\cup \{\infty\}$ be a proper, convex, and lower semi-continuous function. Then, the subdifferential of $f$, $\partial f$ is a maximal monotone mapping from $X$ to $X^*$.
\end{lemma}

\begin{remark}\label{r4}
	Let $E$ be a real Banach space with dual $E^*$ and let $f:X\rightarrow \mathbb{R}\cup \{\infty\}$ be a proper, convex, and lower semi-continuous function. It well known that $f$ posses a subdifferential $\partial f:E\to 2^{E^*}$. Moreover, $0\in \partial f(u^*)$ if and only if $u^*$ is a minimizer of $f$.
\end{remark}
\noindent The theorem is given as the following:

\begin{theorem}\label{thmb3.2}
	Let $E = L_p, 1 < p \le 2$. Also, let $f:E \to [0,\infty)$ be a convex, bounded, and lower semi-continuous Fr\`{e}chet differentiable functional such that $0 \in R(\partial f).$  For arbitrary $x_1 \in E$, define the iterative sequence $\{x_n\}_{n=1}^{\infty}$ by:
	\begin{eqnarray}\label{3b.3}
	x_{n+1} = J^{-1}(Jx_n - \alpha_n\phi_{n} - \alpha_n \theta_nJx_n),\;\;\phi_{n}\in  (\partial f)x_n, \; n \ge 1,
	\end{eqnarray}
	where $\{\alpha_n\}_{n=1}^{\infty}$ and $\{\theta_n\}_{n=1}^{\infty}$ are acceptably paired sequences in $(0,1)$ such that there exists $\gamma > 0$ with $\alpha_n \le \theta_n \gamma \; \forall \; n \ge 1.$ Then, the sequence $\{x_n\}_{n=1}^{\infty}$ converges strongly to the unique point $x^* = \Pi_{(\partial f)^{-1}(0)}0$ of $(\partial f)^{-1}(0)$ satisfying $\langle Jx^*, x - x^* \rangle \le 0$ for all $x \in (\partial f)^{-1}(0).$
\end{theorem}
\begin{proof}
The subdifferential of $f $, i.e., $\partial f$, is maximal monotone by Lemma \ref{opt}. We conclude the proof by employing Theorem \ref{thm3.2} and Remark \ref{r4}.

\end{proof}

\subsection{Application to variational inequality problems}
\vskip0.2truecm
\noindent Let $ E$ be a real Banach space with a dual space $E^*$ and $C$ be a nonempty closed and convex subset of $E$. Let $T:C\subset E\to E^*$ be an arbitrary nonlinear map with $ C\subset D(T)$ and $ f\in E.$  The point $ x^*\in C$ is said to be a solution to the \textit{variational inequality problem} $ \mathrm{VI}(T, f, C)$ if there exists a $ y\in Tx^{*}$ such that
\begin{equation}\label{key1}
\langle t-x^{*}, y-f\rangle \geq 0 ~\textup{for all}~ t\in C.
\end{equation}
\vskip0.3true cm
\noindent We denote the set of solutions of (\ref{key1}) by $SOL(T, f, C).$ The origin of variational inequality problems can be traced to Stampacchia \cite{stamp}, who proposed and investigated them in 1964. This kind of problem, which plays a critical role in nonlinear analysis, is also related to fixed point problems, zeros of nonlinear operators, complementarity problems, and convex minimization problems (see, for example, \cite{Nnakwe20, Saewan13, Thong18c}).

\vskip0.2truecm\noindent
The subsequent result of Rockafella is very crucial in the proof of the theorem in this subsection.
\begin{lemma}[\cite{rock}]\label{rkf}
	Let $C$ be a nonempty closed and convex subset of a real Banach space $E$ with a dual space $E^*$, and let $T:E\to 2^{E^*}$ be a monotone and  hemicontinuous mapping. Then, the mapping $A:E\to 2^{E^*}$ defined by
	\begin{eqnarray}
	Ax=\begin{cases}
	Tx+N_{C}x\qquad \mbox{if} \qquad x\in C,\\
	\emptyset \qquad\qquad\qquad \mbox{if} \qquad x \not\in C,
	\end{cases}
	\end{eqnarray}
	where $N_{C}x=\{u\in E^*:\big<x-v,u\big>\geq 0\;\;\textup{for all}\;\;v\in E\}$, is maximal monotone. Moreover, $0\in Au^*$ if and only if $u^*\in SOL(T, 0, C)$.
\end{lemma}
\noindent Now, we  employ Theorem \ref{thm3.2} to estimate a solution of a VIP by proving the following theorem. 

\begin{theorem}\label{tChmv3.2}
	Let $E = L_p, 1 < p \le 2$. Also, let $T:E \to E^*$ be a bounded, monotone and hemicontinuous mapping such that $ SOL\,(T,0,C) \neq \emptyset.$  For arbitrary $x_1 \in E$, define the iterative sequence $\{x_n\}_{n=1}^{\infty}$ in $ C $ by
	\begin{eqnarray}\label{3Cvc.3}
	x_{n+1} = J^{-1}\left[Jx_n - \alpha_n (Tx_{n}+\beta_{n}) - \alpha_n \theta_nJx_n \right],\quad \beta_n\in\mathrm{N}_{C}(x_{n}),\; n \ge 1,
	\end{eqnarray}
	where $\{\alpha_n\}_{n=1}^{\infty}$ and $\{\theta_n\}_{n=1}^{\infty}$ are acceptably paired sequences in $(0,1)$ such that there exists $\gamma > 0$ with $\alpha_n \le \theta_n \gamma \; \forall \; n \ge 1.$ Furthermore, assume that the sequence $ \{\beta_n\} $ is bounded. Then, the sequence $\{x_n\}_{n=1}^{\infty}$ converges strongly to the unique point $x^* = \Pi_{SOL\,(T,0,C)}0$ of $ SOL\,(T,0,C)$ satisfying $\langle Jx^*, x - x^* \rangle \le 0$ for all $x \in SOL\,(T,0,C).$
\end{theorem}
\begin{proof}
	The proof follows from  Lemma \ref{rkf} and Theorem \ref{thm3.2}.
\end{proof}

\subsection{Application to J-fixed points}
\begin{theorem}\label{thmv3.2}
	Let $E = L_p, 1 < p \le 2$, and $T:E \to E^*$ be a $ J $-pseudocontractive and bounded map such that
	$(J –T)$ is maximal monotone and $ F_{E}^{J}(T)=\{u\in E:Ju= Tu\}\neq \varnothing.$  For arbitrary $x_1 \in E$, define the iterative sequence $\{x_n\}_{n=1}^{\infty}$ by
	\begin{eqnarray}\label{3vc.3}
	x_{n+1} = J^{-1}\left[ (1-\alpha_n)Jx_n - \alpha_n Tx_n - \alpha_n \theta_nJx_n \right] , \; n \ge 1,
	\end{eqnarray}
	where $\{\alpha_n\}_{n=1}^{\infty}$ and $\{\theta_n\}_{n=1}^{\infty}$ are acceptably paired sequences in $(0,1)$ such that there exists $\gamma > 0$ with $\alpha_n \le \theta_n \gamma \; \forall \; n \ge 1.$ Then, the sequence $\{x_n\}_{n=1}^{\infty}$ converges strongly to the unique point $x^* = \Pi_{(J–T)^{-1}(0)}0$ of $(J–T)^{-1}(0)$ satisfying $\langle Jx^*, x - x^* \rangle \le 0$ for all $x \in (J–T)^{-1}(0).$
\end{theorem}
\begin{proof}
Since the map $ A $ is monotone if and only if $ T=(J-A)$ is $ J$--pseudocontractive, the zeros of $ A$ correspond to the $ J$--fixed points of $ T$ (see e.g., \cite{coeu}). Re-writing $(J-T)$ as $ A $ in (\ref{3vc.3}) and applying Theorem \ref{thm3.2} yields the result.
\end{proof}

\section{Numerical Illustrations}
\noindent In this section, we present some numerical examples to show the accuracy and effectiveness of Algorithms  (\ref{3.3}), (\ref{3v.3}), and (\ref{3b.3}) on $L_{p}([0,1])$, $ 1<p\leq 
2$. Numerical experiments were carried out on Matlab R2013a version. All programs were run on a PC with Intel(R) Core(TM)2 Duo CPU and 3GB RAM. All figures were plotted using the $ \loglog $ plot command, while integrals were approximated using the $\trapz$ command on $ \mathrm{MATLAB}$ over the interval $ \left[0,1\right]$, being partitioned into $ 100 $ subintervals. For all experiments, we shall suppose that $ \frac{1}{p}+\frac{1}{q}=1$, where $p=\frac{3}{2}$. 
\vskip0.25cm\noindent
\begin{example}\textbf{(Approximation of zero of Maximal monotone operators)}\label{exam1}
	Let $ A:L_{p}^{\mathbb{R}}([0,1])\longrightarrow L_{q}^{\mathbb{R}}([0,1])$ be defined by 
	\begin{equation}\label{keky}
	(Af)(t)=(t+1)f(t)\quad \textup{for all}  \;t\in [0,1].
	\end{equation}
	Obviously, $ A $ is maximal monotone and bounded. Observe that the normalized duality map\\ $J:L_{p}([0,1])\longrightarrow {L_{q}([0,1])}$ and its inverse $J^{-1}:L_{q}([0,1])\longrightarrow {L_{p}([0,1])}$ can be defined by (see, for instance, \cite{Alber93})
	\begin{eqnarray}
	J(f)= ||f||_{p}^{p-2}f|f|^{2-p}\quad \textup{for all}\;\;f\in L_{p}([0,1])\label{k11ey1}
	\end{eqnarray}
	and
	\begin{equation}\label{k1eya51}
	J^{-1}(g)= ||g||_{q}^{2-q}g|g|^{q-2}\quad \textup{for all}\;\;g\in L_{q}([0,1]),
	\end{equation}
	respectively. Moreover, the norm and inner product on $L_{p}([0,1])$ are defined, respectively, as
	\[ ||\phi||_{p}=\left(\int_{[0,1]}|\phi(t)|^{p}dt \right)^{\frac{1}{p}} \quad\textup{and}\quad\left\langle \phi,\varphi  \right\rangle =\int_{0}^{1}\phi(t)\varphi(t)dt, \]
	for all $ \phi,\varphi\in L_{p}([0,1]).$ The initial points used here are presented in Table \ref{aqw21c12}, while the stopping criteria is given by $ ||x_{n}-x_{n-1}||_{p}<\mathbb{TOL},$ where $\mathbb{TOL}$ represents the tolerance error. The $\mathbb{TOL}$ used in each experiment is also portrayed in Table \ref{aqw21c12}, whereas, $\mathbb{NFE}$ denotes the number of function evaluations. The control sequences $\{\alpha_{n}\}$ and $ \{\theta_{n}\}$ for this example are given by equation (\ref{seq}), where $ n:=i+1,$ $ i\geq 1.$\\

	\noindent	Figure \ref{fig:prob1_61a131} and the second part of Table \ref{aqw21c12} show the numerical results for Algorithm (\ref{3.3}) when the maximal monotone operator is defined by (\ref{keky}), with diverse choices of initial points and $ \mathbb{TOL}=e$--$6.$ Moreover, Figure \ref{fig:prob1_613123} and the first part of Table \ref{aqw21c12} illustrate the strong convergence of our algorithm using several $\mathbb{TOL}$s. These confirm the efficiency and accuracy of the proposed method. 
	
	\vskip1.5pt
	\begin{table}[!ht]
		\centering
		\caption{Computational results and time for Example \ref{exam1}}
		\vskip-10pt
		\begin{tabular}[t]{ccccccc}
			\toprule
			$ x_{1} $&$ \mathbb{TOL} $&$ \mathbb{NFE} $&&$ ||x_{\mathbb{NFE}}-x_{\mathbb{NFE}-1}||_{p} $&&Time (secs)\\
			\midrule
			$ \frac{1}{1+t^{2}} $& $ e $--$ 3 $ &$ 15  $ &&  $8.4352\times 10^{-4} $&&$1.6066\times 10^{-2} $\\
			&$ e $--$ 6 $&$112$ &&  $ 9.7360\times 10^{-7} $&&$6.4984\times 10^{-2}$  \\
			&$ e $--$ 9 $&$1,114 $  &&  $9.9867\times 10^{-10} $&&$5.6356\times 10^{-1}$ \\	
			&$ e $--$ 12 $& $ 12,266 $   &&  $9.9987\times 10^{-13}$&&$5.9324$ \\
			&$ e $--$ 15 $& $ 142,097$   &&  $9.9999\times 10^{-16}$&&$69.6167$ \\
			
			\midrule
			$ \mathbb{TOL} $&$ x_{1} $&$ \mathbb{NFE} $&&$ ||x_{\mathbb{NFE}}-x_{\mathbb{NFE}-1}||_{p} $&&Time (secs)\\
			\midrule
			$ e $--$6$&$e^{t} $  &$ 128 $ &&  $ 9.9784\times 10^{-7} $&&$8.0637\times 10^{-2} $\\
			&$ t^{2}+1 $&$125 $  &&  $9.7860\times 10^{-7} $&&$7.2729\times 10^{-2}$  \\
			&$ \cos(t)e^{-t} $& $ 101$   &&  $9.9679\times 10^{-7}$&&$5.8967\times 10^{-2}$ \\
			&$ \frac{1}{1+t\sin(t)} $& $ 112 $   &&  $9.8481\times 10^{-7}$&&$6.5506\times 10^{-2}$ \\
			
			\bottomrule
		\end{tabular}
		\label{aqw21c12}
	\end{table}%
	
	\vskip-10pt
	\begin{figure}[!h]
		\centering
		\begin{minipage}{.5\textwidth}
			\centering
			\includegraphics[width=\textwidth]{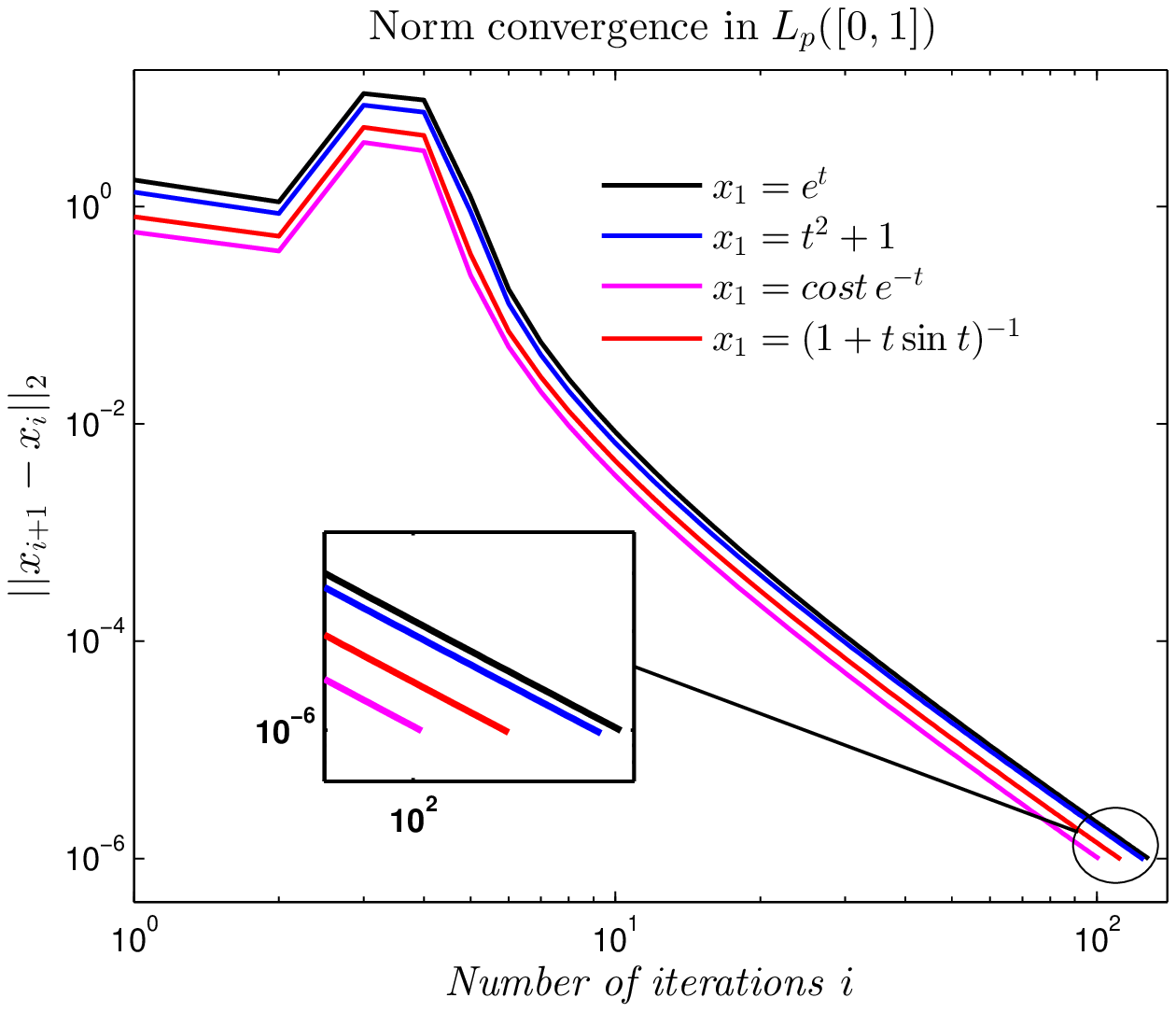}
			\vskip-13pt
			\caption{$ \loglog $ plots in Table \ref{aqw21c12}}
			\label{fig:prob1_61a131}
		\end{minipage}%
		\begin{minipage}{0.5\textwidth}
			\centering
			\includegraphics[width=\textwidth]{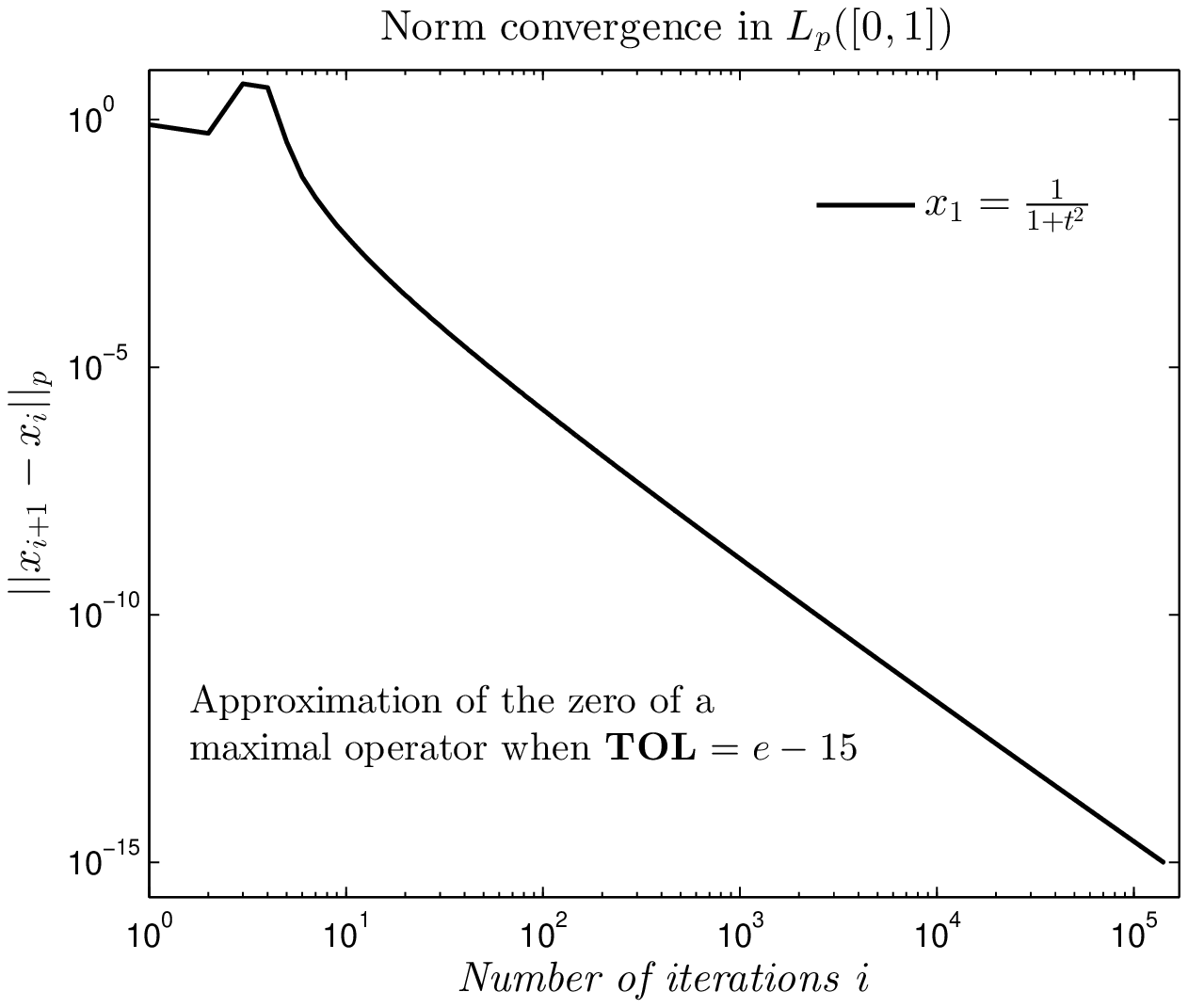}
			\vskip-13pt
			\caption{$ \loglog $ plots in Table \ref{aqw21c12}}
			\label{fig:prob1_613123}
		\end{minipage}
		\vskip-10pt
	\end{figure}
\end{example}
\newpage

\begin{example}\textbf{(Estimation of minimizers)}\label{exam2}
	\noindent	Define a functional $ f:L_{p}([0,1]) \longrightarrow \mathbb{R} $ by $ f(x)=||x||_{p}$, for all $ x\in  L_{p}([0,1])$. The subdifferential of $ f $ at $ x\in L_{p}([0,1])$ is given by 
	\begin{equation*}
	\partial f(x)=\begin{cases}
	\left\lbrace \dfrac{x}{||x||_{p}}\right\rbrace  ,\quad x\neq \mathbf{0},\\[10pt]
	\overline{B}(0,1),\quad\;\; x= \mathbf{0},
	\end{cases}
	\end{equation*} 
	where $ \mathbf{0} $ is the zero function and $ \overline{B}(0,1) $ is the closed unit ball in $ L_{p}([0,1])$. Then, the map $ F:=\partial f $ is maximal monotone and bounded. The normalized duality map and its inverse are defined in (\ref{k11ey1}) and (\ref{k1eya51}), respectively. Moreover, the initial points and $ \mathbb{TOL}$s are presented in Table \ref{aqw2c1c12}, where acceptably paired sequences $\{\alpha_{n}\}$ and $ \{\theta_{n}\}$ are defined by equation (\ref{seq}), with $ n:=i+1,$ $ i\geq 1.$ \\
	
	\noindent	As illustrated in Figure \ref{fig:pcrcob1_61a131} and the second part of Table \ref{aqw2c1c12}, the sequence generated by Algorithm (\ref{3b.3}) converges strongly to the zero of $ \partial ||\cdot||_{p},$ which is also a minimizer of $||\cdot||_{p}$, under different choices of initial points. Choosing $ x_{1}=\frac{1}{1+t^{2}}$, we infer from Table \ref{aqw2c1c12} and Figure \ref{fig:ccprob1_613123} that the sequence $ \{x_{n}\},$ under various $ \mathbb{TOL}$s, converges to the minimizer of $||\cdot||_{p}.$

	\vskip-1.5pt
	\begin{table}[!ht]
		\centering
		\caption{Computational results and time for Example \ref{exam2}}
		\vskip-10pt
		\begin{tabular}[t]{ccccccc}
			\toprule
			$ x_{1} $&$ \mathbb{TOL} $&$ \mathbb{NFE} $&&$ ||x_{\mathbb{NFE}}-x_{\mathbb{NFE}-1}||_{p} $&&Time (secs)\\
			\midrule
			$ \frac{1}{1+t^{2}} $& $ e $--$ 1 $ &$ 47  $ &&  $9.9459\times 10^{-2} $&&$0.0427$\\
			&$ e $--$ 2 $&$465$ &&  $ 9.9863\times 10^{-3} $&&$0.3459$  \\
			&$ e $--$ 3 $&$4,643 $  &&  $9.9971\times 10^{-4} $&&$3.3063$  \\	
			&$ e $--$ 4 $& $ 10,339 $   &&  $9.6881\times 10^{-5}$&&$7.1476$ \\
			\midrule
			$ \mathbb{TOL} $&$ x_{1} $&$ \mathbb{NFE} $&&$ ||x_{\mathbb{NFE}}-x_{\mathbb{NFE}-1}||_{p} $&&Time (secs)\\
			\midrule
			$ e $--$ 2 $&$e^{t} $  &$ 586 $ &&  $ 9.9833\times 10^{-3} $&&$0.4492$\\
			&$ t^{2}+1 $&$586 $  &&  $9.9844\times 10^{-3} $&&$3.3063$  \\	
			&$ \cos(t)e^{-t} $& $ 586$   &&  $9.9818\times 10^{-3}$&&$0.4416$ \\
			&$ \frac{1}{1+t\sin(t)} $& $ 465 $   &&  $9.9853\times 10^{-3}$&&$0.3530$ \\
			\bottomrule
		\end{tabular}
		\label{aqw2c1c12}
	\end{table}%
	\vskip-10pt
	\begin{figure}[!h]
		\centering
		\begin{minipage}{.5\textwidth}
			\centering
			\includegraphics[width=\textwidth]{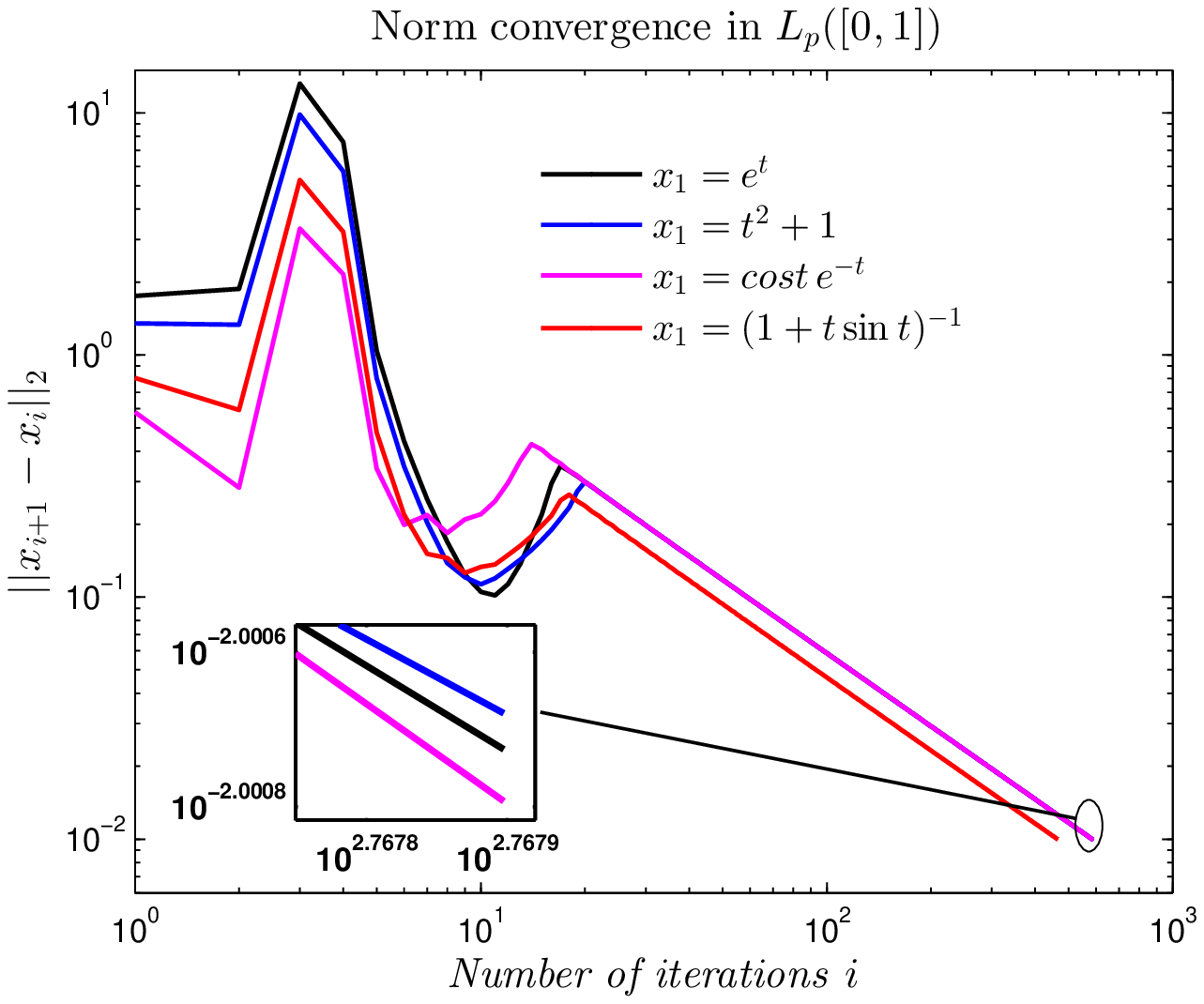}
			\vskip-13pt
			\caption{$ \loglog $ plots in Table \ref{aqw2c1c12}}
			\label{fig:pcrcob1_61a131}
		\end{minipage}
		\begin{minipage}{0.5\textwidth}
			\centering
			\includegraphics[width=\textwidth]{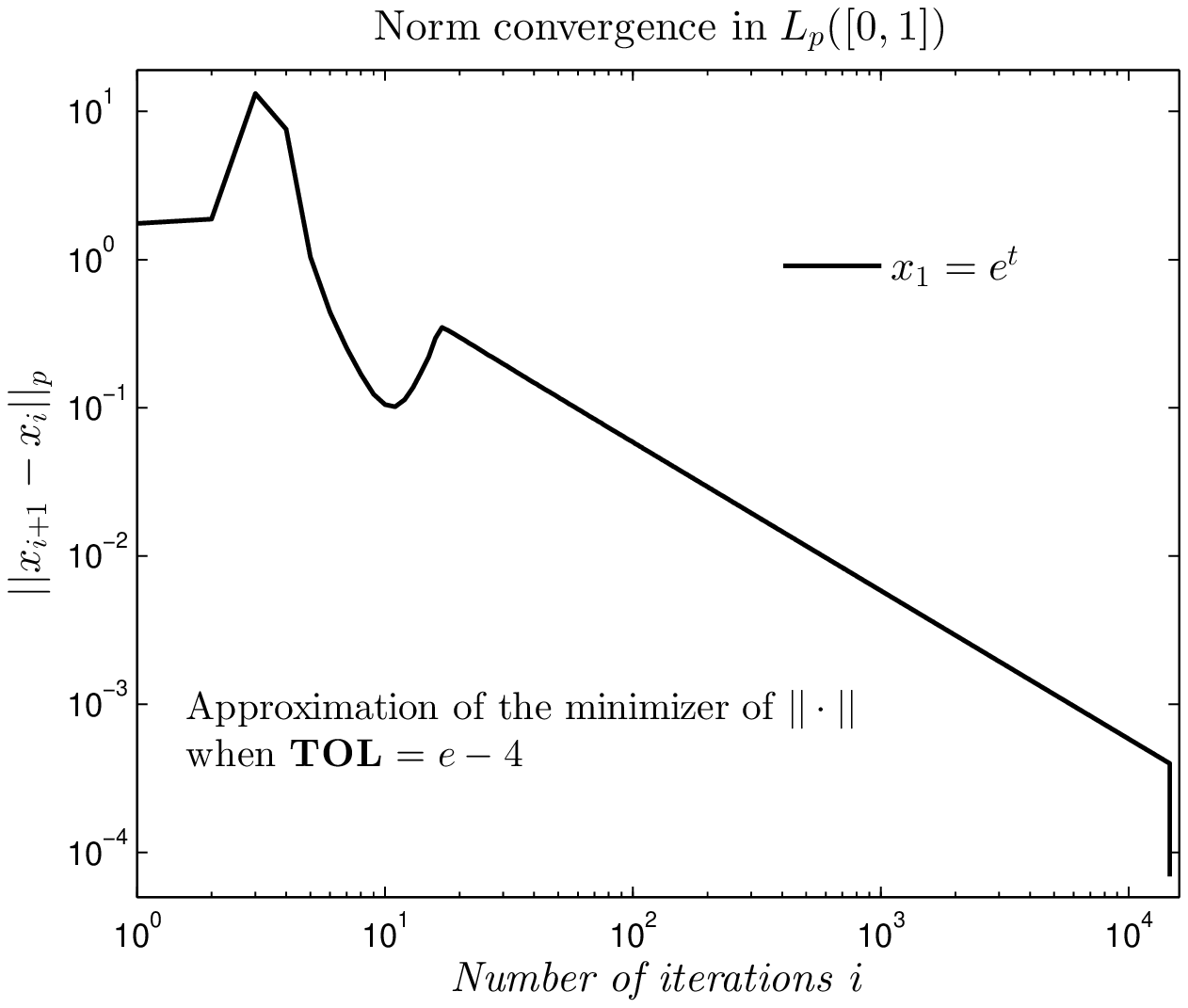}
			\vskip-13pt
			\caption{$ \loglog $ plots in Table \ref{aqw2c1c12}}
			\label{fig:ccprob1_613123}
		\end{minipage}
	\end{figure}

\end{example}
\newpage

\begin{example}\textbf{(Approximation of Hammerstein integral solutions)}\label{exam3}
	Define $ F,K:L_{p}^{\mathbb{R}}([0,1])\longrightarrow L_{p}^{\mathbb{R}}([0,1])$ by $ (Fu)(t)=(t+1)u(t)$ and $ (Ku)(t)=u(t)$, for all $ t\in [0,1]$, respectively. Then, it is clear that $ F $ and $ K $ are maximal monotone and bounded. The initial points and $ \mathbb{TOL}$s for this example are presented in Table \ref{aqw21cc12}. For the algorithm's stopping criteria, we have set it to satisfy both of these: $ ||u_{n}-u_{n-1}||_{p}<\mathbb{TOL}$ and $ ||v_{n}-v_{n-1}||_{p}<\mathbb{TOL}.$ The aforementioned acceptably paired sequences $\{\alpha_{n}\}$ and $ \{\theta_{n}\}$ used in the previous examples are also used here.\\
	
	\noindent	As depicted in Figures \ref{fig:pcrob1_61a131} and \ref{fig:cprob1_613123}, as well as the second portion of Table \ref{aqw21cc12}, the sequences $ \{u_{n}\} $ and $ \{v_{n}\} $ generated by Algorithm (\ref{3v.3}) converge to $u^*=\mathbf{0}$ and $v^*=\mathbf{0}$, respectively. Using different tolerance errors, we present the $ \loglog$ plots of $ ||u_{n}-u_{n-1}||_{p}$ and $ ||v_{n}-v_{n-1}||_{p}$ in Figures \ref{fig:pccrob1_61a131}--\ref{fig:cvprcob1_613123} as well as their numerical values in Table \ref{aqw21cc12}, where the initial points $ u_{1}=\frac{1}{1+t^{2}} $ and $v_{1}=\frac{1}{1+t\sin(t)}$. Hence, we conclude that our algorithm successfully estimates the solution of Hammerstein integral equations.
	
	\vskip-1.5pt
	\begin{table}[!ht]
		\centering
		\caption{Computational results and time for Example \ref{exam3}}
		\vskip-10pt
		\begin{tabular}[t]{lcccccccc}
			\toprule
			&$ \mathbb{TOL} $&$ \mathbb{NFE} $&&$ ||u_{\mathbb{NFE}}-u_{\mathbb{NFE}-1}||_{p} $&&$ ||v_{\mathbb{NFE}}-v_{\mathbb{NFE}-1}||_{p} $&&Time (secs)\\
			\midrule
			$ u_{1}=\frac{1}{1+t^{2}} $& $ e $--$ 3 $ &$ 25  $ &&  $9.6154\times 10^{-4} $&&  $1.8714\times 10^{-4} $&&$2.4732\times 10^{-2} $\\
			$v_{1}=\frac{1}{1+t\sin(t)} $&$ e $--$ 6 $&$247$ &&  $ 7.9082\times 10^{-7} $&& $ 9.8339\times 10^{-7} $&&$2.5633\times 10^{-1}$  \\
			&$ e $--$ 9 $&$3,231 $  &&  $1.0182\times 10^{-10} $&&  $9.9992\times 10^{-10} $&&$3.1548$ \\	
			&$ e $--$ 12 $& $ 39,934 $   &&  $9.9997\times 10^{-13}$ &&  $7.2162\times 10^{-13}$&&$41.260$ \\

			\midrule
			$ \mathbb{TOL} $&$ u_{1} $&$ \mathbb{NFE} $&&$ ||u_{\mathbb{NFE}}-u_{\mathbb{NFE}-1}||_{p} $&&$ ||v_{\mathbb{NFE}}-v_{\mathbb{NFE}-1}||_{p} $&&Time (secs)\\
			\midrule
			$ e $--$6$&$e^{t} $  &$  242$ &&  $ 9.9425\times 10^{-7} $&&  $9.5546\times 10^{-7} $&&$2.4282\times 10^{-1} $\\
			&$ t^{2}+1 $&$246 $  &&  $9.4360\times 10^{-7} $&&  $9.9829\times 10^{-7} $&&$2.5453\times 10^{-1}$  \\
			&$ \cos(t)\,e^{-t} $& $ 247$   &&  $5.6191\times 10^{-7}$&&  $9.9832\times 10^{-7}$&&$2.5403\times 10^{-1}$ \\
			&$ e^{-t} $& $ 209 $   &&  $9.9707\times 10^{-7}$&&  $8.3985\times 10^{-7}$&&$2.3012\times 10^{-1}$ \\

			\bottomrule
		\end{tabular}
		\label{aqw21cc12}
	\end{table}%
	
	\vskip-10pt
	\begin{figure}[!h]
		\centering
		\begin{minipage}{.5\textwidth}
			\centering
			\includegraphics[width=\textwidth]{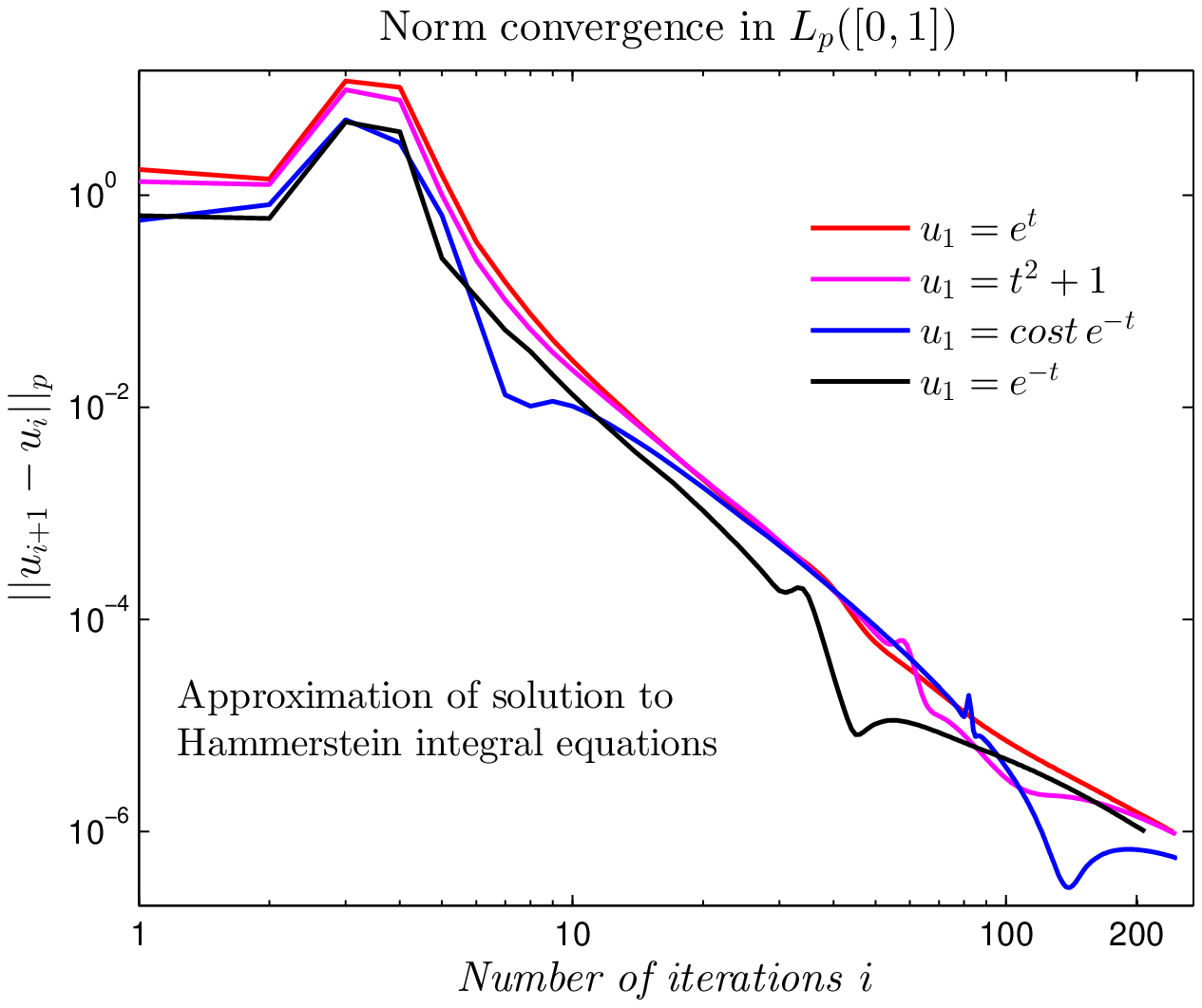}
			\vskip-13pt
			\caption{$\loglog $ plots in Table \ref{aqw21cc12}}
			\label{fig:pcrob1_61a131}
		\end{minipage}%
		\begin{minipage}{0.5\textwidth}
			\centering
			\includegraphics[width=\textwidth]{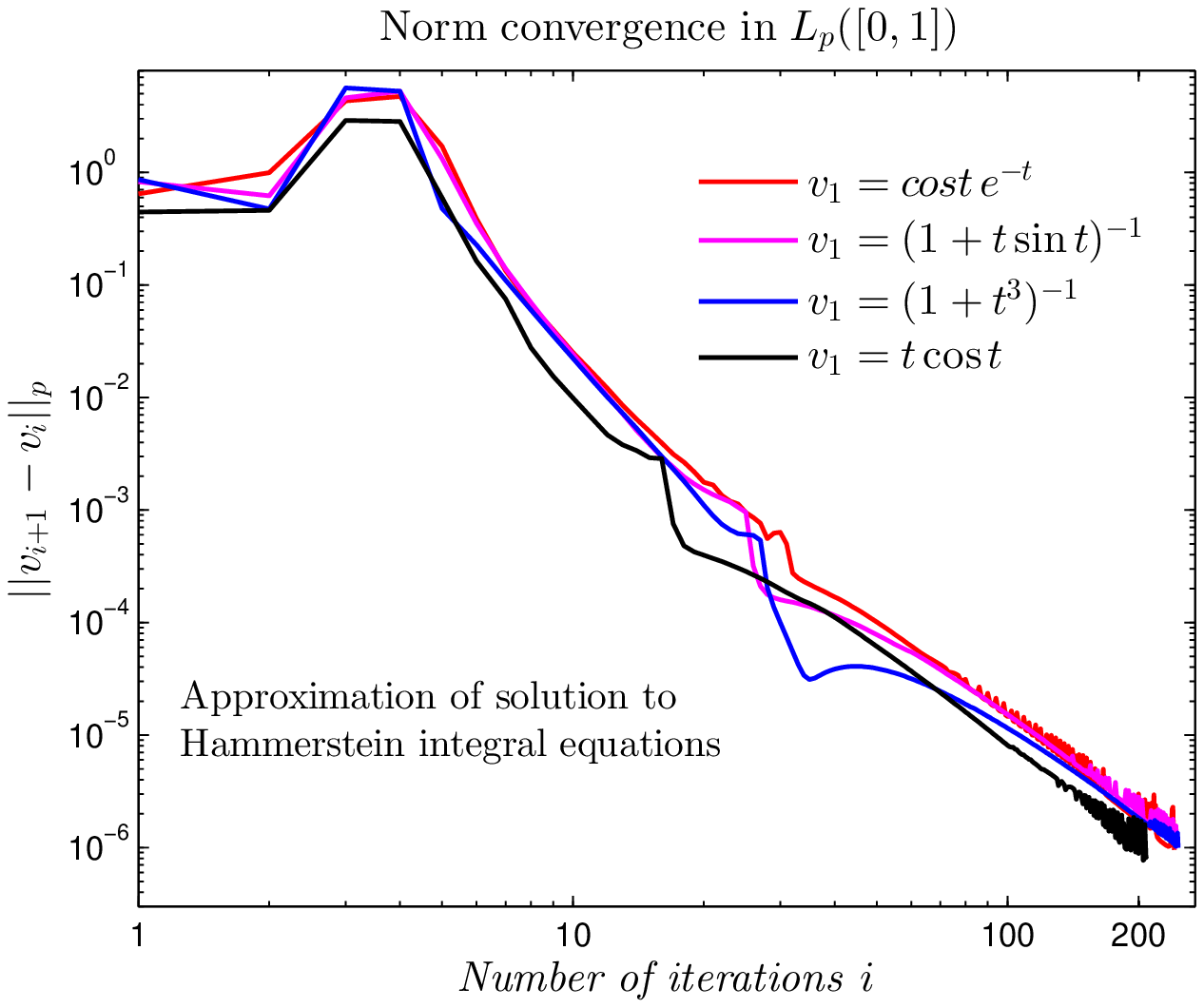}
			\vskip-13pt
			\caption{$\loglog $ plots in Table \ref{aqw21cc12}}
			\label{fig:cprob1_613123}
		\end{minipage}
	\end{figure}

	\vskip-10pt
	\begin{figure}[!h]
		\centering
		\begin{minipage}{.5\textwidth}
			\centering
			\includegraphics[width=\textwidth]{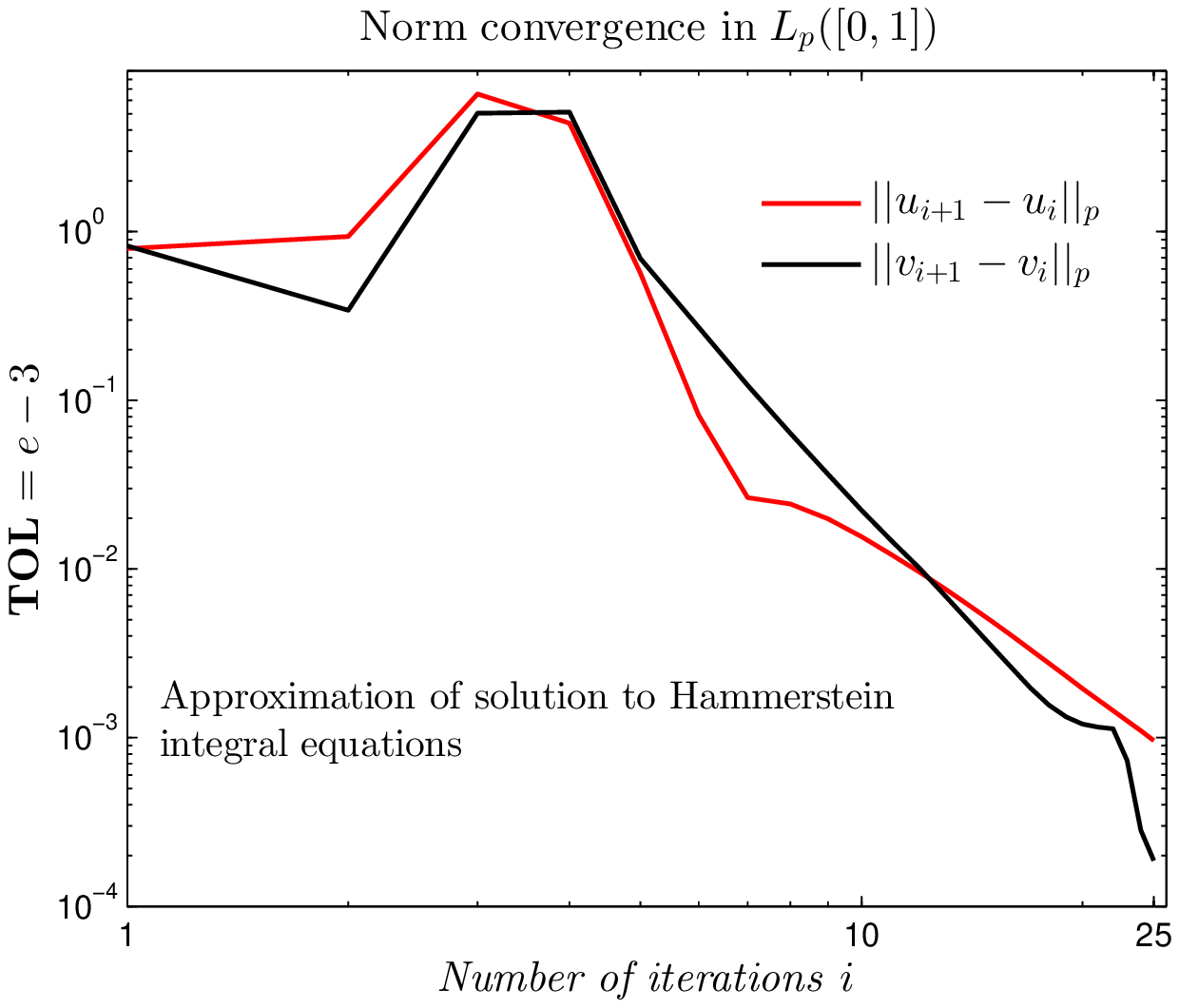}
			\vskip-13pt
			\caption{$\loglog $ plots in Table \ref{aqw21cc12}}
			\label{fig:pccrob1_61a131}
		\end{minipage}%
		\begin{minipage}{0.5\textwidth}
			\centering
			\includegraphics[width=\textwidth]{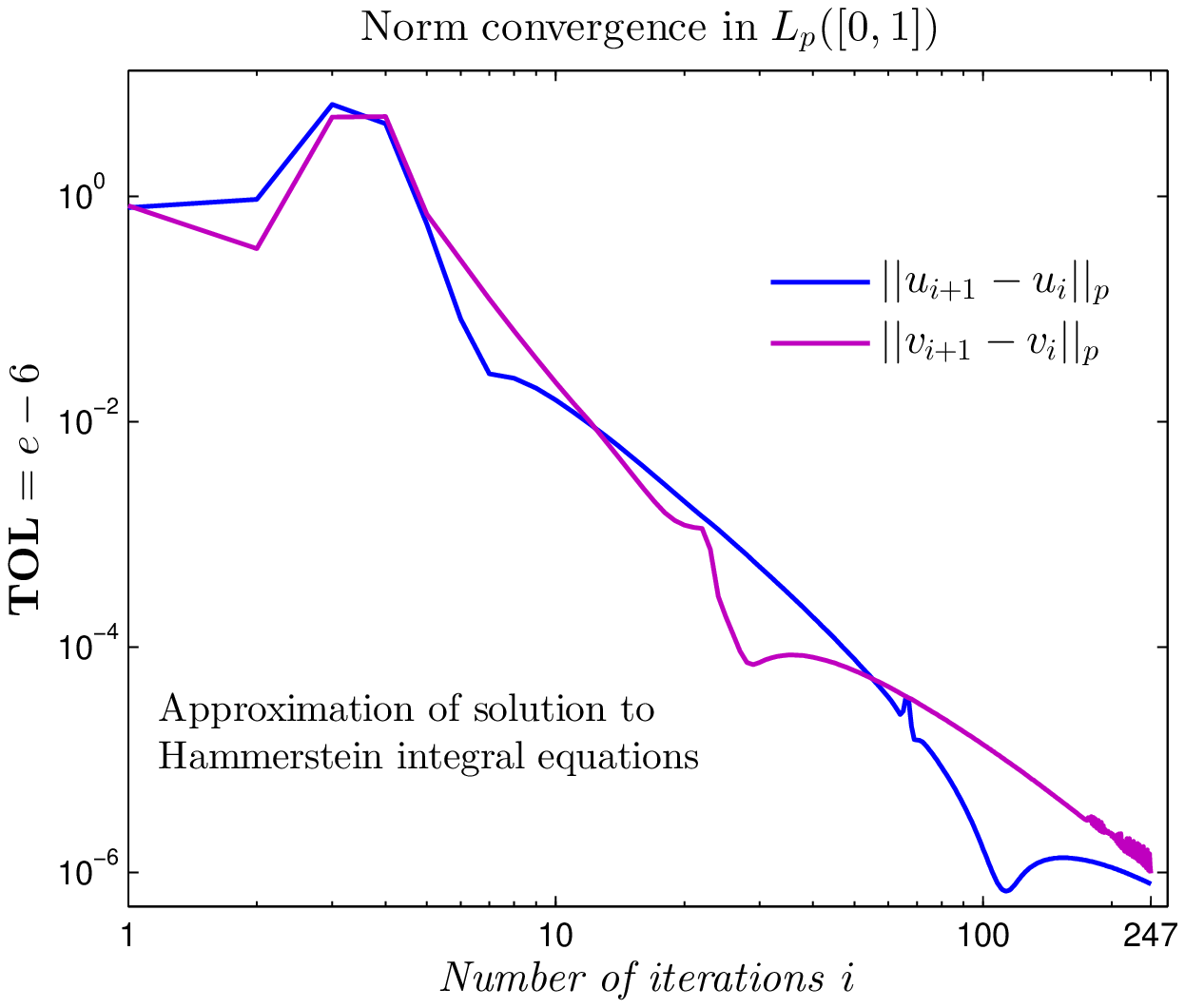}
			\vskip-13pt
			\caption{$\loglog $ plots in Table \ref{aqw21cc12}}
			\label{fig:cprcob1_613123}
		\end{minipage}
		
	\end{figure}
	\vskip-1.5pt
	
	\vskip-10pt
	\begin{figure}[!h]
		\centering
		\begin{minipage}{.5\textwidth}
			\centering
			\includegraphics[width=\textwidth]{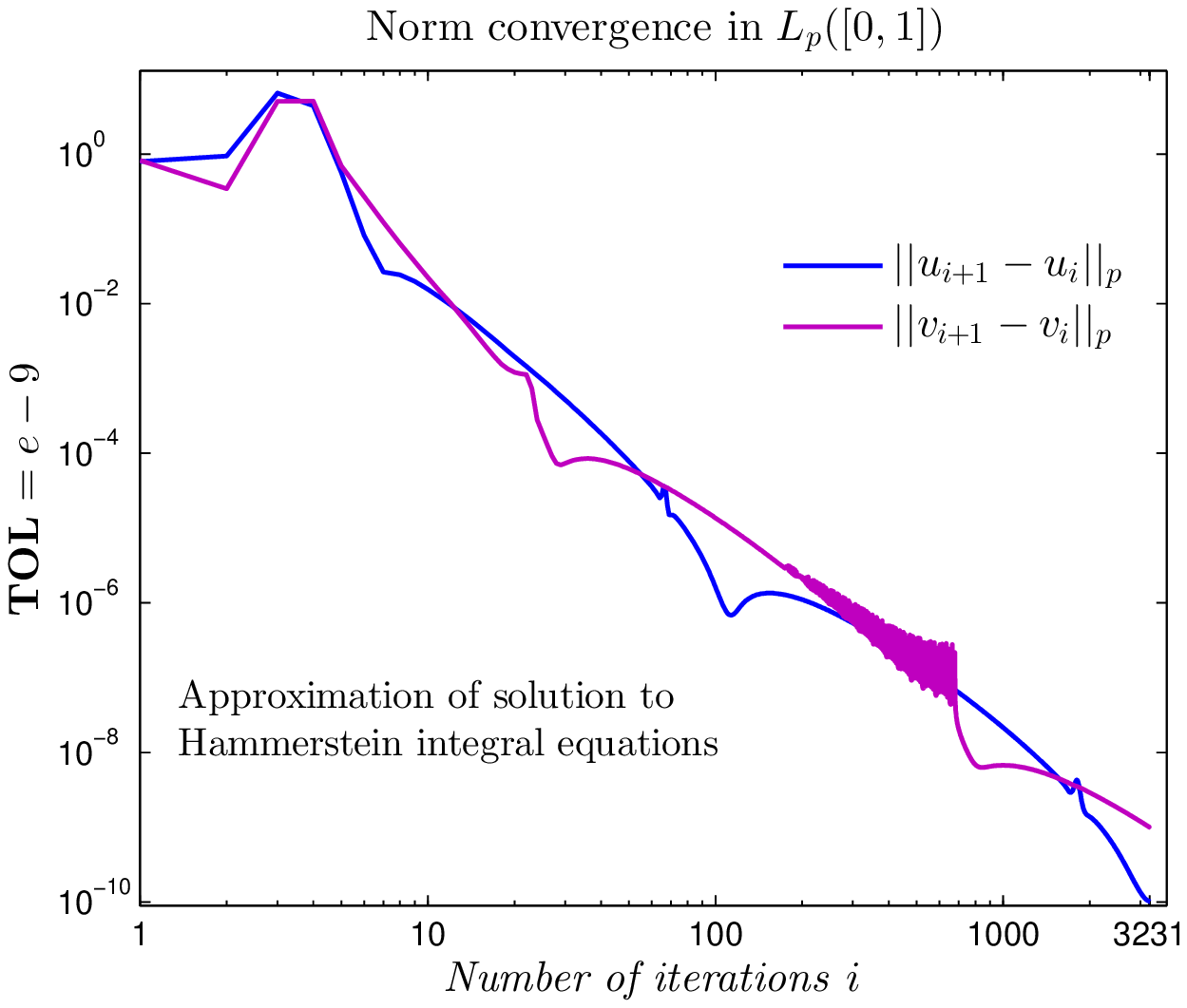}
			\vskip-13pt
			\caption{$\loglog $ plots in Table \ref{aqw21cc12}}
			\label{fig:pccrovb1_61a131}
		\end{minipage}%
		\begin{minipage}{0.5\textwidth}
			\centering
			\includegraphics[width=\textwidth]{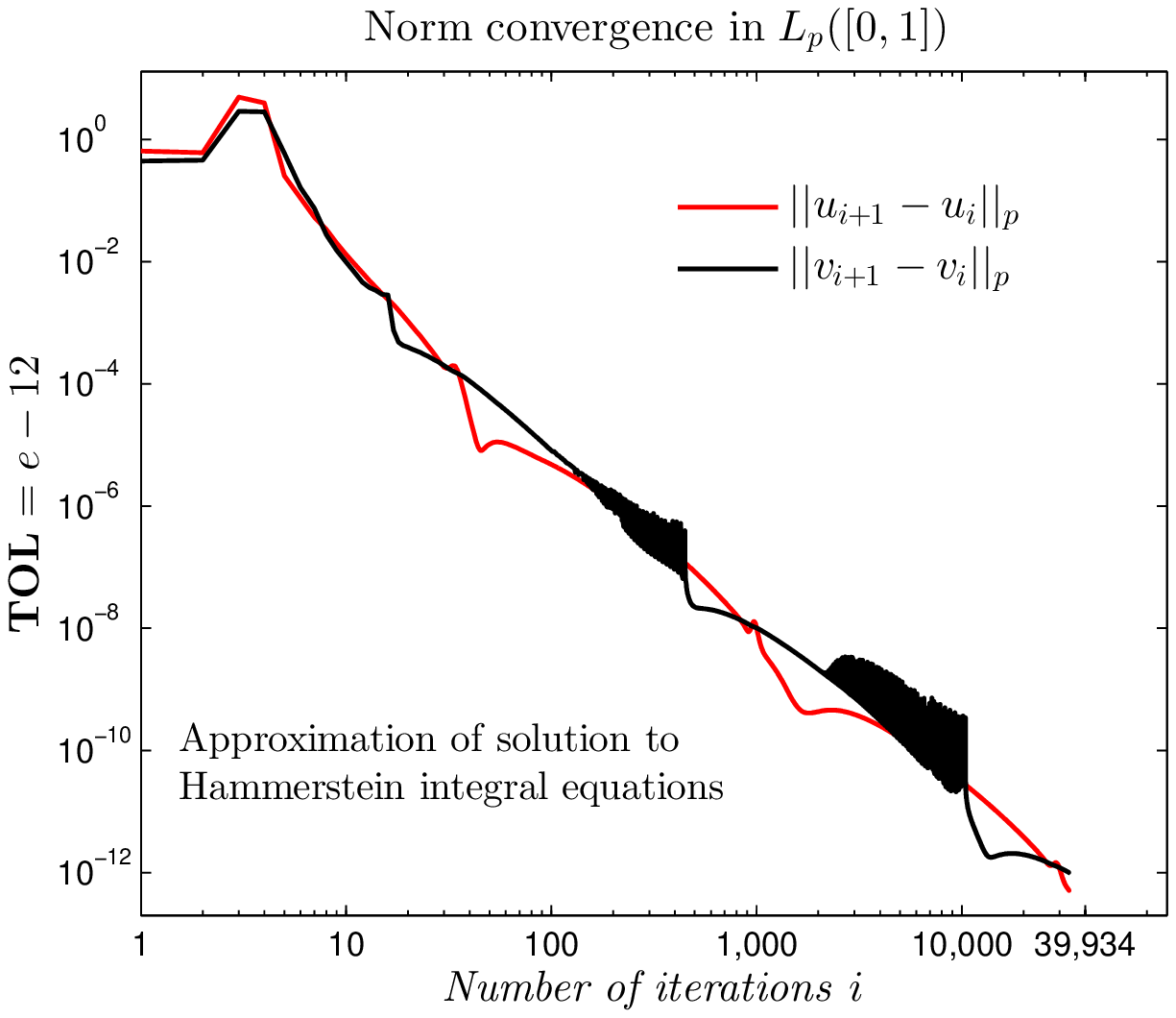}
			\vskip-13pt
			\caption{$\loglog $ plots in Table \ref{aqw21cc12}}
			\label{fig:cvprcob1_613123}
		\end{minipage}
		
	\end{figure}
	
\end{example}

\end{document}